\documentclass[noinfoline]{imsart}
\usepackage{amsmath,amssymb,amsthm,booktabs}
\usepackage{mathrsfs}
\usepackage{amsmath,amssymb,amsthm,enumerate,framed,graphicx,array,color,multirow,mathrsfs,amsrefs}

\usepackage[utf8]{inputenc}
\usepackage[colorlinks,citecolor=blue,urlcolor=blue]{hyperref}

\usepackage{bm}
\usepackage{euscript}
\usepackage{graphicx}

\usepackage{multicol}
\usepackage[usenames,dvipsnames,svgnames,table]{xcolor}

\usepackage{a4wide}

\usepackage{mathrsfs}
\usepackage{euscript}

\startlocaldefs
\numberwithin{equation}{section} \theoremstyle{plain}
\newtheorem{theorem}{Theorem}[section]
\newtheorem{lemma}{Lemma}[section]

\newtheorem{proposition}{Proposition}[section]

\newtheorem{definition}{Definition}
\newtheorem{remark}{Remark}[section]
\endlocaldefs

\allowdisplaybreaks

\allowdisplaybreaks[4]
\numberwithin{equation}{section}

\def\X{\boldsymbol X}

\def\bC{\mathbb{C}}
\def\bN{\mathbb{N}}
\def\bR{\mathbb{R}}
\def\bE{\mathbb{E}}\def\E{\mathbb{E}}
\def\bP{\mathbb{P}}

\def\cF{\mathcal{F}}

\def\cP{\mathcal{P}}
\def\cL{\mathcal{L}}

\def\fM{\mathfrak M}

\def\diag{\mathrm{diag}}

\def\T{\textsf{T}}

\def\1{\mathbf 1}

\newcommand{\rectangle}{{%
  \ooalign{$\sqsubset\mkern2mu$\cr$\mkern2mu\sqsupset$\cr}%
}}

\begin{document}

\title{On eigenvalues of the Brownian sheet matrix}
\runtitle{Eigenvalues of Brownian sheet matrix}

  \begin{aug}
    \author{\fnms{~ Jian} \snm{Song}\ead[label=e1]{txjsong@hotmail.com}}\footnote{J. Song is partially supported by 
    Shandong University grant 11140089963041 and National Natural Science Foundation of China grant 12071256.},
    \author{\fnms{Yimin} \snm{Xiao}\ead[label=e2]{xiaoy@msu.edu}}\footnote{Y. Xiao is partially supported by 
    NSF grant DMS-1855185.}
    \and
    \author{\fnms{Wangjun} \snm{Yuan}\ead[label=e3]{ywangjun@connect.hku.hk}}
    
    \affiliation{Shandong University and  The University of Hong Kong}
    \runauthor{J. Song,   Y. Xiao \& W. Yuan}

    \address{Research Center for Mathematics and Interdisciplinary Sciences, Shandong University, Qingdao, Shandong, 266237, China;
    and School of Mathematics, Shandong University, Jinan, Shandong, 250100, China\\
      \printead{e1}}
    
    \address{ Department of Statistics and Probability, Michigan State University, A-413 Wells Hall, East Lansing,
MI 48824, U.S.A.\\
      \printead{e2}
   }

    \address{
      Department of Mathematics, 
      The University of Hong Kong\\
      \printead{e3}
    }
  \end{aug}



\date{}

\begin{abstract}  We derive a system of stochastic partial differential equations satisfied by the eigenvalues 
of  the symmetric  matrix whose entries are the Brownian sheets. We prove that the sequence 
$\left\{L_{d}(s,t), (s,t)\in[0,S]\times [0,T]\right\}_{d\in\bN}$ of  
empirical spectral measures of the rescaled matrices is tight on $C([0,S]\times [0,T], \mathcal P(\bR))$ and hence is convergent as $d$ goes to infinity by Wigner's semicircle law.  We also obtain PDEs which are satisfied by the high-dimensional limiting measure.
\end{abstract}

\maketitle

{
\hypersetup{linkcolor=black}
 \tableofcontents 
}

\section{Introduction}

This paper concerns the eigenvalues of the Brownian sheet matrix  ${\boldsymbol X}= \{\X(s,t),0\le s,t<\infty\} $, which is  
a symmetric-matrix-valued process with entries $X_{ij}$ for $1\le i, j\le d$ given by
\begin{equation}\label{e:X}
X_{ij}(s,t)=\begin{cases}
b_{ij} (s,t), &i < j,\\
\sqrt{2} b_{ii}(s,t), & i=j,
\end{cases}
\end{equation}
where $b = \left\{ b_{ij}(s,t), 0\le s, t<\infty\right \}_{1\le i \le j\le d }$ is a family of independent  Brownian sheets. 

After the fundamental work \cite{Wigner55} which established the celebrated Wigner's semicircle law, Brownian motion as a 
one-parameter stochastic process  was introduced into random matrix theory by Dyson \cite{Dyson1962}.  Since then, there 
has been fruitful literature on the Dyson Brownian motion which is the system of eigenvalues of symmetric Brownian matrix %has been well studied
(see, e.g.  \cite{Cepa1997, Chan1992, Rogers1993, anderson2010, EY17} and the references therein), in which It\^o's calculus 
has played a key role. By studying the high-dimensional limit of the empirical measures of the Dyson Brownian motion, one can 
provide a dynamical proof for Wigner's semicircle law (see, e.g., \cite{anderson2010}). The Dyson Brownian motion is also closely 
related to interacting particle systems, and the equation (known as the McKean-Vlasov equation) satisfied by its limiting empirical measure 
appears naturally in the study of propagation of chaos for large systems of interacting particles (see, e.g., \cite{BO19, JW18, Serfaty20}).

Multiparameter stochastic processes (or random fields) are a natural extension of one-parameter processes, they arise naturally in statistical 
mechanics (e.g. Brownian sheet appears in the Ising model \cite{KM86} and interacting particle systems \cite{KT88}), and systematic theories 
have been developed (see, e.g., \cite{CW75,K02} and the references therein). Motivated by the close connection between random matrix 
theory and  interacting particle systems, it is natural to develop theories for random matrix with entries being random fields. Recently, 
the problem on the collision of eigenvalues of symmetric (Hermitian) matrix whose entries are independent Gaussian fields  was investigated  
in \cite{Jaramillo2018, SXY20}, which  to our best knowledge are the only literature on random matrix whose entries are random fields. 

Another motivation for studying the Brownian sheet matrix $\X$ is from free probability theory.  As shown in  \cite{Voi91, VDN92}, many theorems and 
concepts in free  probability have  classical probability analogs, and furthermore free probability is closely connected with random matrix theory. In 
particular, free Brownian motion can be viewed as the high-dimensional limit of rescaled Brownian motion matrix which is define by \eqref{e:X} with 
$b$ being a family of independent Brownian motions.  Stochastic calculus for free Brownian motion was developed in \cite{BS98}. Free fractional 
Brownian motion arose naturally   in \cite{NT14} when studying the central limit theorem for long-range dependence time series in free probability, 
and the stochastic calculus  was developed in \cite{DS19}. It was shown in \cite{Pardo2016} that free fractional Brownian motion is the 
high-dimensional limit of empirical measures of the eigenvalues of rescaled fractional Brownian motion matrices.
%by studying the high-dimensional limit of the empirical measures. 
We remark that the free Brownian motion and the free fractional Brownian motion in  \cite{BS98,NT14,Pardo2016,DS19} are one-parameter 
stochastic processes, and we believe that our study of the Brownian sheet matrix in this paper will provide a useful building block for constructing 
free random fields. 

%Stochastic differential equations for the eigenvalues of fractional Brownian matrix were obtained in \cite{NP14}, and  the results were extended in \cite{Pardo2017, SYY20}.   

In the present paper we shall derive a system of stochastic partial differential equations \eqref{eq-spde''} for the eigenvalue processes 
of the Brownian sheet matrix $\X$ given by \eqref{e:X}, obtain the tightness of the spectral empirical measures (Theorem \ref{Thm-tightness}), 
and show that  the limit measure satisfies a McKean-Vlasov equation \eqref{e:MV} and a Burgers' equation \eqref{e:Burgers'}.  
We briefly explain the structure of the paper below. 

 Though the Brownian sheet is a simple multivariable extension 
of standard Brownian motion, 
%with parameter extended from the line to the plain, 
the stochastic calculus for the Brownian sheet that one needs for deriving the stochastic partial differential equations for 
the eigenvalues of the Brownian sheet matrix turns out to be highly non-trivial and cannot be adapted directly from the 
classical It\^o calculus. In Section \ref{sec:stochastic-calculus}, we follow the approach of Cairoli and Walsh in \cite{CW75} 
and develop stochastic calculus tools for the multi-dimensional Brownian sheet on the plane for our purpose.  The main results in   
this section are Theorems \ref{Thm-Green formula} and \ref{Coro-Green formula} which are multi-dimensional Green's formulas.

In Section \ref{sec:SPDE},  by applying classical It\^o's formula together with Green's formulas (Theorems \ref{Thm-Green formula} 
and \ref{Coro-Green formula} ), 
we derive the system of stochastic partial differential equations \eqref{eq-spde''} for the eigenvalues of the Brownian sheet matrix 
$\X$. Compared with the following system of SDEs for the classical Dyson Brownian motion: for $1\le i\le d$,
\begin{equation}\label{Eq:Dyson}
d\lambda_i(t) = \sqrt 2 dW_i(t) + \sum_{j\neq i}\frac1{\lambda_i(t)-\lambda_j (t)}dt,
\end{equation}
where $W=(W_1, \dots, W_d)$ is a standard $d$-dimensional Brownian motion, we remark that 
eq. \eqref{eq-spde''} bears some resemblance to \eqref{Eq:Dyson} but has several extra high-order terms.
% than classical Dyson Brownian motion which is governed by the following system of SDEs: for $1\le i\le d$,
%\begin{equation}
%d\lambda_i(t) = \sqrt 2 dW_i(t) + \sum_{j\neq i}\frac1{\lambda_i(t)-\lambda_j (t)}dt,
%\end{equation}
%where $W=(W_1, \dots, W_d)$ is a standard $d$-dimensional Brownian motion. 

In Section \ref{sec:limit}, we study the high-dimensional limit of empirical distributions for the eigenvalue processes of $\X$. 
In Section \ref{sec:tightness}, we establish the tightness of the set of empirical spectral 
measures which are viewed as $C([0,S]\times[0,T], \mathcal P(\bR))$-valued random elements (see Theorem 
\ref{Thm-tightness}). This guarantees that every sequence of the empirical spectral measures has a subsequence 
which converges weakly. The tightness together with the classical Wigner's semicircle law implies the existence and 
uniqueness of the high-dimensional limit of the empirical spectral measures (see Theorem \ref{Thm-limit measure}).  
In Section \ref{sec:PDE},  we derive partial differential equations \eqref{e:MV} and \eqref{e:Burgers'} that are satisfied 
by the limiting measure, by using the property of the semicircle distribution.

Finally, in Appendix \ref{sec:appendix} we provide some results in matrix analysis which are needed in our analysis.

\section{Stochastic calculus for the Brownian sheet}\label{sec:stochastic-calculus}
In this section, we shall apply the stochastic calculus on the plane developed in \cite{CW75} to derive  Green's formula 
for  the multi-dimensional Brownian sheet,  which is a key ingredient for studying SPDEs for the eigenvalues 
in Section \ref{sec:SPDE}.  

\subsection{Some preliminaries on stochastic calculus on the plane}
In this subsection, we recall from Cairoli and 
Walsh \cite{CW75} some preliminaries for stochastic 
calculus on the plane. %We refer to \cite{CW75} for more details. 

Define the partial order ``$\prec$'' on $\bR^2$ as follows. For any $(s_1, t_1), \, (s_2, t_2) \in \bR^2$,
\[(s_1,t_1)\prec (s_2, t_2), \text{ iff } s_1\le s_2, \, t_1\le t_2,\]
and write 
\[(s_1,t_1)\prec \prec(s_2, t_2), \text{ iff } s_1< s_2, \, t_1< t_2.\]

Let $(\Omega, \mathcal G, \bP)$ be a probability space and let the filtration $\cF=\{\cF_z, z\in \bR_+^2\}$ be a family of 
sub-$\sigma$-field of $\mathcal G$ satisfying
\begin{enumerate}
\item $\cF_z\subset \cF_{z'}$ if $z\prec z'$;
\item $\cF_0$ contains all null sets of $\mathcal G$;
\item for each $z, \cF_z=\bigcap\limits_{z\prec\prec z'}\cF_{z'}$;
\item for each $z, \cF_{z}^1$ and $\cF_{z}^2$ are conditionally independent given $\cF_z$.
\end{enumerate}
Here, for $z=(s,t)\in \bR_+^2$, 
\begin{align*}
\cF_{z}^1=\cF_{s\infty}:= \underset{v}\vee \cF_{sv};~~~
\cF_{z}^2=\cF_{\infty t}:= \underset{u}{\vee}\cF_{ut}.
\end{align*}
In particular, the augmented filtration generated by a finite family of independent Brownian sheets satisfies the 
above conditions.

Let $Y=\{Y_z, z\in \bR_+^2\}$ be a process such that for each $z$ the random variable $Y_z$ is integrable.  
We recall the definitions of martingale, strong martingale, weak martingale, and increasing process relative to 
$\cF$ in \cite{CW75}. 
\begin{definition}
$Y$ is a \emph{martingale} if
\begin{enumerate}
\item $Y$ is adapted;
\item $\bE[Y_{z'}|\cF_z]=Y_z,$ for each $z\prec z'$.
\end{enumerate}
\end{definition}

Suppose $z=(s,t)$ and $z'=(s',t')$ such that $z\prec\prec z'$. We denote by $(z, z']$ the rectangle 
$(s,s']\times(t,t']$. The increment of $Y$ over the rectangle $(z, z']$ is 
\[Y((z,z'])=Y_{s't'}-Y_{st'}-Y_{s't}+Y_{st}.\]
\begin{definition}~
\begin{itemize}
\item[(a)] $Y$ is a \emph{weak martingale} if 
\begin{enumerate}
\item $Y$ is adapted;
\item $\bE[Y((z,z'])|\cF_z]=0$ for each $z\prec\prec z'.$
\end{enumerate}
\end{itemize}
\begin{itemize}
\item[(b)] $Y$ is an \emph{$i$-martingale} ($i=1,2$) if 
\begin{enumerate}
\item $Y$ is $\cF_z^i$-adapted;
\item $\bE[Y((z,z'])|\cF_z^i]=0$ for each $z\prec\prec z'.$
\end{enumerate}
\end{itemize}
\begin{itemize}
\item[(c)] $Y$ is a \emph{strong martingale} if 
\begin{enumerate}
\item $Y$ is adapted;
\item $Y$ vanishes on the axes;
\item $\bE[Y((z,z'])|\cF_z^1\vee \cF_z^2]=0$ for each $z\prec\prec z'.$
\end{enumerate}
\end{itemize}
\end{definition}

\begin{definition}
 $Y$ is an \emph{increasing process} if 
 \begin{enumerate}
 \item $Y$ is right-continuous and adapted;
 \item $Y_z=0$ on the axes;
 \item $Y(A)\ge0$ for each rectangle $A\subset \bR_+^2$. 
 \end{enumerate}
\end{definition}

Let $M=\{M_z, z\in\bR_+^2\}$ be a martingale relative to $\cF$. Then $M$ is both a 1-martingale and 2-martingale, 
i.e., $\{M_{s0}, \cF_{s0}^1, s\in\bR_+\}$ and $\{M_{0t}, \cF_{0t}^2, t\in\bR_+\}$ are martingales. The converse is also true. 

 Now we assume that $M$ is a square integrable martingale. By \cite[Theorem 1.5]{CW75}, there exists an increasing 
 process $\langle M\rangle$  such that $M^2-\langle M\rangle$ is a weak martingale.  For each fixed $t$, let 
 $\{[M]^1_{st}, s\in\bR_+ \}$ be the unique increasing process which is predictable relative to $\{\cF_{st}, s\in \bR_+\}$
 such that $\{M^2_{st}-[M]^1_{st}, s\in \bR_+\}$ is a martingale.  Similarly, one can define $[M]^2$. As pointed by  
 \cite[p.121]{CW75}, for a strong martingale $M$, either $[M]^1$ or $[M]^2$ can serve as the process $\langle M\rangle$. 
Furthermore, by \cite[Theorem 1.9]{CW75},  if either $\cF$ is generated by the Brownian sheet or $M$ has finite fourth 
moment, then $[M]^1=[M]^2$, and hence we can choose $\langle M \rangle=[M]^1=[M]^2$. As a consequence,   
for any fixed $t$, $\{M^2_{st}-\langle M \rangle_{st}, s\in \bR_+\}$ is a martingale, and similarly, for any fixed $s$, 
$\{M^2_{st}-\langle M \rangle_{st}, t\in \bR_+\}$ is a martingale.   As in \cite{CW75}, we shall use $d_s \langle M \rangle_{st}$ 
($d_t \langle M \rangle_{st}$, resp.) to denote the differential of $\langle M\rangle $ with respect to $s$ ($t$, resp.).

For two square integrable martingales $M$ and $N$, we denote by $\langle M, N\rangle$ any process which is the 
difference of two increasing processes such that $MN-\langle M, N\rangle$ is a weak martingale. %Noting
%that $MN=\frac12\left((M+N)^2-M^2-N^2\right)$, 
One can choose, for instance, 
\begin{equation}\label{e:quad-cov}
\langle M, N\rangle =\frac12\big(\langle M+N \rangle- \langle M\rangle -\langle N\rangle \big).
\end{equation}
Define $[M,N]^i=\frac12\left([M+N]^i-[M]^i-[N]^i\right)$ for $i=1,2.$ Then either $[M,N]^1$ or $[M,N]^2$ can serve as the process 
$\langle M,N \rangle.$ Two martingales $M$ and $N$ are said to be  \emph{orthogonal} if $MN$ is a weak martingale, 
and we write $M\perp N$.

For $p\ge 1$, let $\fM^p$  denote  the set  of right-continuous martingales $M=\{M_z, z\in \bR_+^2\}$ such that 
$M=0$ on the axes and $\bE[|M_z|^p]<\infty$ for all $z\in \bR_+^2$. Let $\fM^p_c$ (resp. $\fM_s^p$) be the set of 
continuous (resp. strong) martingales in $\fM^p$. Similarly, let $\fM^p(z_0)$ (resp. $\fM^p_c(z_0), \fM^p_s(z_0)$) 
be the set of right-continuous (resp. continuous, strong) martingales $M=\{M_z, z\prec z_0\}$ such that $M_z=0$ 
on the axes and $\bE[|M_z|^p] <\infty$ for all $z\prec z_0$. 

Below we recall some results which will be used in our proofs. 
\begin{theorem}{\cite[Theorem 1.2]{CW75}} \label{Thm-BDG 2dim}
Let $\{M_z:z \in \bR_+^2\}$ be a right-continuous martingale. Then for $p > 1$,
\begin{align*}
	\bE \left[ \sup_z |M_z|^p \right]
	\le \left( \dfrac{p}{p-1} \right)^{2p} \sup_z \bE \left[ |M_z|^p \right].
\end{align*}
\end{theorem}
For any $z\in \bR_+^2$, we denote the rectangle $(0, z]$ by $R_z$. We also fix $z_0\in \bR_+^2$. 

\begin{theorem}{\cite[Proposition 1.6]{CW75}} \label{Thm-prop1.6}
Let $M,N\in \fM^2(z_0)$. Then 
\begin{enumerate}
\item $\bE[(MN)(D)|\cF_z]=\E[M(D) N(D)|\cF_z]$ for each rectangle $D=(z, z']\subset R_{z_0}$;
\item $M\perp N$  iff   ~$\E[M(D)N(D)|\cF_z]=0$ for each rectangle $D=(z, z']\subset R_{z_0}$. 
\end{enumerate}
\end{theorem}

\begin{theorem}{\cite[Proposition 1.8]{CW75}} \label{Thm-prop1.8}
If $M\in \fM_s^2(z_0)$, then $[M]^i$ is the unique $\cF_z^i$-predictable increasing process such that 
for $i=1,2$,
\begin{align*}
	\bE \left[ M(D)^2|\cF_z^i \right]
	= \bE \left[ (M^2)(D)|\cF_z^i \right]
	= \bE \left[ [M]^i(D)|\cF_z^i \right]
\end{align*}
for each rectangle $D = (z,z'] \subseteq R_{z_0}$. Consequently, for $M,N\in \fM_s^2(z_0)$, noting
 that $MN=\frac12\left((M+N)^2-M^2-N^2\right)$, we have for $i=1,2$,
\begin{align*}
	\bE \left[ M(D)N(D)|\cF_z^i \right]
	= \bE \left[ (MN)(D)|\cF_z^i \right]
	= \bE \left[ [M,N]^i(D)|\cF_z^i \right]
\end{align*}
 
\end{theorem}

\begin{theorem}{\cite[Theorem 1.9]{CW75}} \label{Thm-prop1.9}
Let  $M\in \fM_s^2$. Assuming either the filtration $\cF$ is generated by the Brownian sheet or $M$ is 
continuous with finite fourth moment, we have $[M]^1 = [M]^2$.
\end{theorem}

%If we further that $M$ is a strong martingale and that $\cF$ is generated by $W$, then $\langle M\rangle$ is unique.

\subsection{On $\psi\cdot MN$ and $J_{MN}$ }\label{sec:J-MN}

Let us recall from \cite[Section 6]{CW75} the notion $J_M$ of a continuous martingale $M\in \mathfrak M_s^4$ 
on $\bR_+^2$.  
Recall the notation $R_{st} = (0,s] \times (0,t]$. By  \cite[Eq.~(6.3)]{CW75}, 
\begin{align*}
	J_M(s_0,t_0)
	=& \int_0^{s_0} M(s,t_0) M(ds,t_0) - \int_{R_{s_0t_0}} M(s,t) dM(s,t)\\
	=& \int_0^{t_0} M(s_0,t) M(s_0,dt) - \int_{R_{s_0t_0}} M(s,t) dM(s,t)\\
	=& \frac12 M^2(s_0, t_0)-\frac12\langle M\rangle_{s_0, t_0}-\int_{R_{s_0t_0}} M(s,t) dM(s,t).
\end{align*}
Heuristically, one has $dJ_M(s,t)= M(s,dt) M(ds,t)$ (see \cite{CW75}). Similarly, for two $\cF$-adapted martingales 
$M$ and $N$,  we introduce the following generalization $J_{MN}$ which induces the measure  $M(s, dt) N(ds,t)$ on $\bR_+^2$, 
\begin{align} \label{def-J}
	J_{MN}(s_0,t_0)
	= \int_0^{s_0} M(s,t_0) N(ds,t_0) - \int_{R_{s_0t_0}} M(s,t) dN(s,t),
\end{align}
assuming that the right-hand side is well-defined.  Clearly we have $J_M=J_{MM}$.  Analogous to $J_M$ in \cite[Theorem 6.1]{CW75}, 
$J_{MN}$ will play a key role in the multi-dimensional Green's formula in the forthcoming Theorems \ref{Thm-Green formula} and 
\ref{Coro-Green formula}.

Similar to \cite {CW75}, we shall represent $J_{MN}$ by a new type of stochastic integral denoted by $\psi\cdot MN$ which will be 
defined in the sequel.  Firstly, we need to introduce another order relation ``$\curlywedge$'' in $\bR_+^2$ which is complementary 
to ``$\prec$'' and plays an essential role in the definition of $\psi\cdot MN$. For $z=(s,t)$ and $z'=(s', t')$, we say $z \curlywedge z'$ 
if $s\le s'$ and $t\ge t'$, and $z{\curlywedge\atop \curlywedge} z'$ if $s<s'$ and $t>t'$.  In the $st$-plane where the $s$-axis is  
horizontal and the $t$-axis is vertical,  $z\curlywedge z'$ means that $z$ is on the upper left of $z'$ in the plane. As a comparison, 
$z\prec z'$ means that $z$ is on the lower left of $z'$. 

%
%{\blue
%Following p.47 on \cite{CW75}, define 
%\[J_{MN}^{n} (z) =\sum_{i,j} N(\varepsilon_{ij}\cap R_z)M(\delta_{ij}\cap R_z).\]
%Then 
%\[J_{MN}^n(\Delta_{ij})=M(\varepsilon_{ij})N (\delta_{ij})\]
%which suggests that $d J_{MN} (s,t) = M(ds, t) N (s,dt).$
%
%You need to check that $J_{MN}^n\to J_{MN}$ (or $J_{MN}^n$ converges to  $J_{N,M}$ defined in \eqref{def-J}).   
%
%
%On the other hand, $J_{MN}$ can be informally denoted by $\Psi\cdot MN$, and following the analysis on $\Psi\cdot MM$ 
%in the previous section, you may show  \[\langle J_{MN} \rangle_{st} = d_s\langle M\rangle_{st}d_t\langle M\rangle_{st}.\]
%}

\begin{proposition} \label{Prop 2.1}
Suppose  $M,N\in \fM_s^2(z_0)$.  Let  $A=(z_A, z_A']$ and $B=(z_B, z_B']$ be two rectangles such that $A\curlywedge B$,
 i.e.,  $z_1\curlywedge z_2$ for all $z_1\in A$ and $z_2\in B$. 

Define the process $X= \{X_z, ~z \in \bR_+^2\}$ by
\begin{align*}
X_z = \xi M(A\cap R_z) N(B \cap R_z), ~z \in \bR_+^2,
\end{align*}
where $\xi$ is bounded and $\cF_{z_A\vee z_B}$-measurable. Then $X$ belongs to $\fM^2(z_0)$, it  is continuous if $M$ is, and 
\begin{align}\label{e:quad-var}
	\langle X \rangle_z
	=\xi^2 \iint_{R_z \times R_z} \1_A(z_1) \1_B(z_2) d[M]^2_{z_1} d[N]^1_{z_2} . 
\end{align}
\end{proposition}

\begin{proof}
We will follow the proof of \cite[Proposition 2.4]{CW75}. 

For $D = (z,z']$ with $z = (s,t) \prec \prec z' = (s',t')$, the increment of $X$ over $D$ is
\begin{equation}\label{e:increment-X}
X(D) = M({\tilde A})N({\tilde B}),
\end{equation} 
where $\tilde A = A \cap (R_{s't'}\backslash R_{s't})$ and ${\tilde B} = B \cap (R_{s't'}\backslash R_{st'})$.

Suppose $z_{\tilde A}$ is the lower-left corner of $\tilde A$. Then both $\xi$ and $N({\tilde B})$ are $\cF_{z_{\tilde A}}^2$-measurable, 
and hence
\begin{align*}
        \bE \left[ X(D)|\cF_z^2 \right]
	= \bE \left[ \left.\bE\left[\xi M({\tilde A})N({\tilde B})|\cF_{z_{\tilde A}}^2\right]\right|\cF_z^2 \right]= \bE \left[\left.
	 \xi N({\tilde B})\bE\left[M({\tilde A})|\cF_{z_{\tilde A}}^2\right]\right|\cF_z^2 \right]= 0.
\end{align*}
Similarly, one can show $\bE \left[ X(D)|\cF_z^1 \right] = 0$. Hence, $X$ is a martingale.

Let $z_{\tilde B}$ be the lower left-hand corner of $\tilde B$, and denote $z_0=z_{\tilde A}\vee z_{\tilde B}$. Then 
$z_A\vee z_B\prec z_0$, and hence $\xi$ is $\cF_{z_0}$-measurable. Thus, by Theorem \ref{Thm-prop1.6},
\[\bE[X^2(D)|\cF_{z}]=\bE\left[\left.\xi^2\bE\left[ M(\tilde A)^2N(\tilde B)^2|\cF_{z_0}\right] \right|\cF_{z}\right]. \]
 Now we have
\begin{align*}
	\bE \left[ M({\tilde A})^2 N({\tilde B})^2|\cF_{z_0} \right] =& \bE \left[ M({\tilde A})^2|\cF_{z_0} \right] \bE \left[ N({\tilde B})^2|\cF_{z_0} \right] \\
	=& \bE \left[\left. \bE \left[ M({\tilde A})^2 | \cF_{z_0}^2 \right] \right|\cF_{z_0} \right] \bE \left[ \bE\left. \left[ N({\tilde B})^2|\cF_{z_0}^1 \right] \right|\cF_{z_0} \right] \\
	=& \bE \left[  [M]^2({\tilde A})  |\cF_{z_0} \right] \bE \left[ [N]^1({\tilde B}) |\cF_{z_0} \right] \\
	=& \bE \left[ [M]^2({\tilde A})[N]^1({\tilde B}) |\cF_{z_0} \right],
\end{align*}
where the first and the last equalities follow from the assumption that $\cF_{z_0}^1$ and $\cF_{z_0}^2$ are conditionally 
independent given  $\cF_{z_0}$, and the third equality follows from  Theorem \ref{Thm-prop1.8}. Thus
$$\bE[X^2(D) - \xi^2 [M]^2({\tilde A})[N]^1({\tilde B})|\cF_z]=0$$ 
and hence $X_z^2-\langle X\rangle_z$ is a weak martingale where $\langle X\rangle_z$ is given by  \eqref{e:quad-var}. 
The proof is concluded.
\end{proof}

With Proposition \ref{Prop 2.1} in mind, we define a new type of stochastic integral denoted by $\psi\cdot MN$, 
following the approach in \cite{CW75}. 

Fix an integer $n$ and $z_0=(s_0, t_0)\in\bR_+^2$. Divide $R_{z_0}$ into rectangles $\rectangle_{i,j} = (z_{i,j},z_{i+1,j+1}]$, 
where $z_{i,j} = (is_0/2^n,jt_0/2^n)$ for $i, j=0,1, \dots, 2^n-1$.   We first define $\psi \cdot MN$ for an indicator function $\psi.$  
If $i,j,k,l$ are positive integers with $1\le i < k \le 2^n$ and $1\le l < j \le 2^n$, i,e. $\rectangle_{i,j}\curlywedge \rectangle_{k,l} $, 
define the so-called indicator function
\begin{align}\label{eq:psi}
	\psi_{ijkl}(z_1,z_2) = \xi \1_{\rectangle_{i,j}}(z_1) \1_{\rectangle_{k,l}}(z_2),
\end{align}
where $\xi$ is bounded and $\cF_{z_{k,j}}$-measurable, 
and define
\begin{align*}
	(\psi_{ijkl} \cdot MN)_z
	= \xi M(\rectangle_{i,j} \cap R_z) N(\rectangle_{k,l} \cap R_z), \ \ \ \ z \in R_{z_0}.
\end{align*}
Then by Proposition \ref{Prop 2.1}, $\psi_{ijkl} \cdot MN$ is a well-defined  square integrable martingale with quadratic variation 
\begin{align*}
	\langle \psi_{ijkl} \cdot MN \rangle_z
	= \iint_{R_z \times R_z} \psi_{ijkl}^2(z_1,z_2) d[ M]^2_{z_1} d[ N]^1_{z_2}, 
\end{align*}
and thus we have the following isometry
\begin{equation}\label{e:iso}
\E[|\psi_{ijkl}\cdot MN|^2] =\E\left[\iint_{R_z \times R_z} \psi_{ijkl}^2(z_1,z_2) d[ M]^2_{z_1} d[ N]^1_{z_2}\right].
\end{equation}

We shall define $\psi\cdot MN$ for a more general class of integrands $\psi$ following the standard approximation procedure. For this purpose, one needs the isometry  \eqref{e:iso} to hold for finite sum  of indicator functions, and it suffices to prove the following equality
\begin{align}\label{e:qua-var'}
	\langle \psi_{ijkl} \cdot MN, \psi_{mpqr} \cdot MN \rangle_z
	= \iint_{R_z \times R_z} \psi_{ijkl}(z_1,z_2)  \psi_{mpqr}(z_1,z_2) d[M]^2_{z_1} d[ N]^1_{z_2}.
\end{align}
Here,  $\psi_{mpqr}(z_1, z_2)=\eta \1_{\rectangle_{m,p}}(z_1) \1_{\rectangle_{q,r}}(z_2)$ with $m<q\le 2^n, r<p\le 2^n$ and  
 $\eta$  being a bounded $\cF_{z_{q,p}}$-measurable random variable. To prove \eqref{e:qua-var'}, we consider the following more general situation.

Suppose  $M,N,M',N' \in \mathcal  \fM_s^2(z_0)$, and let $(A,B)$ and $(A',B')$ be two pairs of rectangles satisfying the conditions 
in Proposition 2.1, i.e., $A\curlywedge B$ and $A'\curlywedge B'$. Furthermore, we assume $A, A', B, B'$ are from the set 
$\{\rectangle_{i,j}, i,j=0, 1, \dots, 2^n-1\}$. Thus, any two of the rectangles $A, A', B, B'$ are either coincide or disjoint. Denote 
$z_0=(z_A\vee z_B)\vee (z_{A'}\vee z_{B'})$. We claim that the following equality holds
\begin{align} \label{e:increment}
	&\bE \left[ \left. M(A)M'(A')N(B)N'(B') \right| \cF_{z_0} \right]
	+ \bE \left[ \left. M(A')M'(A)N(B)N'(B') \right| \cF_{z_0} \right] \nonumber \\
	&\qquad + \bE \left[ \left. M(A)M'(A')N(B')N'(B) \right| \cF_{z_0} \right]
	+ \bE \left[ \left. M(A')M'(A)N(B')N'(B) \right| \cF_{z_0} \right] \nonumber \\
	&= \bE \left[ \left. \Big( M(A)M'(A') + M(A')M'(A) \Big) \Big( N(B)N'(B') + N(B')N'(B) \Big) \right| \cF_{z_0} \right] \nonumber \\
	&= 4 \bE \Big[ \left. [M,M']^2(A \cap A') [N,N']^1(B \cap B') \right| \cF_{z_0} \Big].
\end{align}

\begin{proof}[\it Proof of (\ref{e:increment})]
The first equality is straightforward. In the following, we shall prove the second equality. 

Recall that the four rectangles $A, A', B, B'$ are either disjoint or coincide; furthermore,  $A\curlywedge B$ and  $A'\curlywedge B'$, 
i.e., $A$ (resp. $A'$) is on the upper left side of $B$ (resp. $B'$).  We prove the second inequality  in \eqref{e:increment} 
by separating the relative locations of $A, A', B',B'$ into four cases. In the following, we denote the lower left corner of a rectangle 
$E$ by $z_E$. 

{\bf Case 1.} If $A$ is on the top of $A'$ and $A\cap A'=\emptyset$, noting that $A$ (resp. $A'$) is to the upper left of $B$ (resp. $B'$),  
we have that $M'(A'), N(B), N'(B')$ are all $\cF_{z_A}^2$-measurable. Since $M$ is a $2$-martingale, we have, noting that 
$\cF_{z_0}\subset \cF_{z_A}^2$
\begin{align*}
	\bE \left[ \left. M(A)M'(A')N(B)N'(B') \right| \cF_{z_0} \right]
	&= \bE \Big[ \left. \bE \left[ \left. M(A)M'(A')N(B)N'(B') \right| \cF_{z_A}^2 \right] \right| \cF_{z_0} \Big] \\
	&= \bE \Big[ \left. M'(A')N(B)N'(B') \bE \left[ \left. M(A) \right| \cF_{z_A}^2 \right] \right| \cF_{z_0} \Big] \\
	&= 0.
\end{align*}
Similarly, for the other terms on the left-hand side of \eqref{e:increment} we also have
\begin{align*}
	&\bE \left[ \left. M(A')M'(A)N(B)N'(B') \right| \cF_{z_0} \right] = 0, \\
	&\bE \left[ \left. M(A)M'(A')N(B')N'(B) \right| \cF_{z_0} \right] = 0, \\
	&\bE \left[ \left. M(A')M'(A)N(B')N'(B) \right| \cF_{z_0} \right] = 0.
\end{align*}
Summing over all the above equalities, we get \eqref{e:increment}.

{\bf Case 2.} If $A'$ is on the top of  $A$ and $A\cap A'=\emptyset$, the proof is the same by considering the $\sigma$-field 
$\cF_{z_{A'}}^2$. If $B$ is to the right (resp. left) of $B'$ with $B\cap B'=\emptyset$, then the proof is also the same by considering 
the $\sigma$-field $\cF_{z_B}^1$ (resp. $\cF_{z_{B'}}^1$).

{\bf Case 3.} Now we only have one situation left: $A$ and $A'$ are  at the same horizontal level, which is on the top of $B$ and 
$B'$, and $B$ and $B'$ are at the same vertical level, which is to the right of $A$ and $A'$. We denote $z_0 := z_A \vee z_B
=z_{A'}\vee z_{B'}$.  Note that $M(A), M'(A'), M(A'), M'(A)$ are $\cF_{z_0}^1$ measurable and $N(B), N'(B'), N(B'), N'(B)$ 
are $\cF_{z_0}^2$ measurable. We have
\begin{align} \label{e:increment'}
	&\quad \bE \left[ \left. \Big( M(A)M'(A') + M(A')M'(A) \Big) \Big( N(B)N'(B') + N(B')N'(B) \Big) \right| \cF_{z_0} \right] \nonumber \\
	&= \bE \Big[ \left. \bE \left[ M(A)M'(A') + M(A')M'(A) \right| \cF_{z_0} \right] \bE \left[ \left. N(B)N'(B') + N(B')N'(B) \right| 
	\cF_{z_0} \right] \Big| \cF_{z_0} \Big],
\end{align}
where the equality follows from the conditional independence of $\cF_{z_0}^1$ and $\cF_{z_0}^2$ given $\cF_{z_0}$.

To compute
\begin{align*}
	\bE \left[ \left. M(A)M'(A') + M(A')M'(A) \right| \cF_{z_0} \right],
\end{align*}
we split it into the following three cases.

(a) If $A = A'$, noting that $\cF_{z_0}^2 = \cF_{z_A}^2$, by Theorem \ref{Thm-prop1.8}, 
	\begin{align} \label{eq-conditional A=A'}
		\bE \left[ \left. M(A) M'(A) \right| \cF_{z_0}^2 \right]
		= \bE \left[ \left. (MM')(A)  \right| \cF_{z_0}^2 \right]
		= \bE \left[ \left. [M,M']^2(A) \right| \cF_{z_0}^2 \right].
	\end{align}

(b) If $A$ and $A'$ are two adjacent disjoint rectangles on the same horizontal level, then $A \cup A'$ is also a rectangle. 
Without loss of generality, we may assume that $A$ is to the left of $A'$, then $z_A$ is also the lower left corner of 
$A \cup A'$. Thus, by Case (a), we have
	\begin{align*}
		&\bE \left[ \left. M(A) M'(A) \right| \cF_{z_0}^2 \right]
		= \bE \left[ \left. [M,M']^2(A) \right| \cF_{z_0}^2 \right], \\
		&\bE \left[ \left. M(A') M'(A') \right| \cF_{z_0}^2 \right]
		= \bE \left[ \left. [M,M']^2(A') \right| \cF_{z_0}^2 \right], \\
		&\bE \left[ \left. M(A \cup A') M'(A \cup A') \right| \cF_{z_0}^2 \right]
		= \bE \left[ \left. [M,M']^2(A \cup A') \right| \cF_{z_0}^2 \right].
	\end{align*}
	Noting that $M(A \cup A') = M(A) + M(A')$, $M'(A \cup A') = M'(A) + M'(A')$ and $[M,M']^i(A \cup A') = [M,M']^i(A) 
	+ [M,M']^i(A')$, we subtract the first two equations from the third one and obtain
	\begin{align} \label{eq-eq-conditional A not= A'}
		\bE \left[ \left. M(A)M'(A') + M(A')M'(A) \right| \cF_{z_0}^2 \right] = 0.
	\end{align}

(c) If $A$ and $A'$ are two non-adjacent rectangles on the same horizontal level, we denote  by $A''$  the rectangle 
between $A$ and $A'$. Note that $A''$ is the union of small rectangles in the set $\{\rectangle_{i, j}, i, j=1, \dots, 2^n\}$. 
By Case (b), we have
	\begin{align*}
		&\bE \left[ \left. M(A)M'(A'') + M(A'')M'(A) \right| \cF_{z_0}^2 \right] = 0, \\
		&\bE \left[ \left. M(A)M'(A'' \cup A') + M(A'' \cup A')M'(A) \right| \cF_{z_0}^2 \right] = 0.
	\end{align*}
Noting that $M'(A'' \cup A') = M'(A'') + M'(A')$, one can subtract the first equality from the second one to obtain 
\eqref{eq-eq-conditional A not= A'}.

Therefore,  summarizing  the three cases (a-c), we can write
\begin{align*}
	\bE \left[ \left. M(A)M'(A') + M(A')M'(A) \right| \cF_{z_0}^2 \right]
	= 2 \bE \left[ \left. [M,M']^2(A \cap A') \right| \cF_{z_0}^2 \right].
\end{align*}
Hence, by taking conditional expectation with respect to the $\sigma$-field $\cF_{z_0}$, we have
\begin{align} \label{eq-quadratic on M A}
	\bE \left[ \left. M(A)M'(A') + M(A')M'(A) \right| \cF_{z_0} \right]
	= 2 \bE \left[ \left. [M,M']^2(A \cap A') \right| \cF_{z_0} \right].
\end{align}

In the same spirit, we can also prove
\begin{align} \label{eq-quadratic on N B}
	\bE \left[ \left. N(B)N'(B') + N(B')N'(B) \right| \cF_{z_0} \right]
	= 2 \bE \left[ \left. [N,N']^1(B \cap B') \right| \cF_{z_0} \right].
\end{align}

Finally, substituting \eqref{eq-quadratic on M A} and \eqref{eq-quadratic on N B} into \eqref{e:increment'}, we have
\begin{align*}
	& \bE \left[ \left. \left( M(A)M'(A') + M(A')M'(A) \right) \left( N(B)N'(B') + N(B')N'(B) \right) \right| \cF_{z_0} \right] \\
	&= 4 \bE \left[ \left. \bE \left[ \left. [M,M']^2(A \cap A') \right| \cF_{z_0} \right] \bE \left[ \left. [N,N']^1(B \cap B') \right| 
	\cF_{z_0} \right] \right| \cF_{z_0} \right] \\
	&= 4 \bE 
	\left[ \left. [M,M']^2(A \cap A') [N,N']^1(B \cap B') \right| \cF_{z_0} \right],
\end{align*}
where the conditional independence of $\cF_{z_0}^1$ and $\cF_{z_0}^2$ given $\cF_{z_0}$ is used again in the last equality. 
This proves \eqref{e:increment}.
\end{proof}

By choosing $M'=M$ and $N'=N$,  eq. \eqref{e:increment} degenerates to
\begin{align} \label{e:increment1}
	\bE \left[ \left. M(A)M(A')N(B)N(B') \right| \cF_{z_0} \right]
	=&  \bE \Big[ \left. [M,M]^2(A \cap A') [N,N]^1(B \cap B') \right| \cF_{z_0} \Big].
\end{align}
%
%(This plays the role of eq. (2.17) in C-W's paper, and allows us to define $\psi\cdot MN$ for linear combination of step function. In particular, $\psi \cdot MN$ for $\psi (z_1, z_2) = 1_{z_1 {\curlywedge\atop \curlywedge} z_2}$ is well defined, and it is also proved to be identical with $J_{MN}$ in our manuscript.)
Now, as in Proposition \ref{Prop 2.1}, we can define 
\begin{equation}\label{e:xx'}
	X_z = \xi M(A\cap R_z) N(B \cap R_z)~ \text{ and }~ X'_z = \xi' M'(A'\cap R_z) N'(B' \cap R_z)
\end{equation}
 for some bounded variables $\xi\in\mathcal F_{z_A\vee z_B}$ and $\xi'\in\mathcal F_{z_{A'}\vee z_{B'}}$. Denote 
 $z_0:=(z_A\vee z_B)\vee (z_{A'}\vee z_{B'})$ and we assume $z_0\prec z=(s,t)$, since otherwise at least one of 
 $X_z$ and $X_z'$ is zero. Let $z'=(s',t')$ be such that $z\prec\prec z'$ and let $D:=(z, z']$. 
	
Assuming $M=M'$ and $N=N'$ in \eqref{e:xx'}, 	following the approach used in the proof of Proposition \ref{Prop 2.1}, 
we can show by \eqref{e:increment},
	\[\E[(XX')(D)|\cF_z] = 	
	\xi\xi' \bE \Big[ \left. [M,M]^2(\tilde A \cap \tilde A') [N,N]^1(\tilde B \cap \tilde B') \right| \cF_{z} \Big], \]
where $\tilde A = A \cap (R_{s't'}\backslash R_{s't})$, ${\tilde B} = B \cap (R_{s't'}\backslash R_{st'})$, and 
$\tilde A' = A' \cap (R_{s't'}\backslash R_{s't})$ and ${\tilde B'} = B' \cap (R_{s't'}\backslash R_{st'})$. This leads to 
\begin{align}\label{e:quad-var1}
	\langle X, X' \rangle_z
	=\xi\xi' \iint_{R_z \times R_z} \1_{A\cap A'}(z_1) \1_{B\cap B'}(z_2) d[M, M]^2_{z_1} d[N, N]^1_{z_2},
\end{align}
and hence \eqref{e:qua-var'} is verified. 

Now we are ready to define $\psi\cdot MN$ for a more general integrand $\psi$.  We say $\psi$ is a simple function if it is a finite sum of $\psi_{ijkl}$ 
given in \eqref{eq:psi}. Let $\mathcal D$ be the $\sigma$-filed on $\bR_+^2\times \bR_+^2\times \Omega$ 
generated by all the simple functions. We call $\mathcal D$ the field of predictable sets. Let $\cL_{MN}^2(z_0)$ 
be the class of all predictable processes such that
\begin{align}
	\bE \left[ \iint_{R_{z_0} \times R_{z_0}} \psi^2(z_1,z_2) d[ M]^2_{z_1} d[ N]^1_{z_2}\right] < \infty.
\end{align}
Then $\cL_{MN}^2(z_0)$ is a Hilbert space with the inner product
\begin{align}\label{e:isometry}
	(\psi, \phi)
	= \bE \left[ \iint_{R_{z_0} \times R_{z_0}} \psi(z_1,z_2) \phi(z_1,z_2) d[ M]^2_{z_1} d[ N]^1_{z_2} \right],
\end{align}
and the simple functions form a dense subset. By \eqref{e:qua-var'} and \eqref{e:isometry}, the mapping 
$\psi\mapsto \psi\cdot MN$ defines an isometry between the set of simple functions and $ \mathfrak M^2(z_0)$.  Then, by a 
standard approximation argument, one can extend the definition of $
\psi\cdot MN$ for each process $\psi\in \cL^2_{MN}(z_0)$.  Furthermore, \eqref{e:qua-var'} also yields for $z\prec z_0$,
\begin{equation}\label{e:qua-var}
\langle \psi\cdot MN, \phi\cdot MN \rangle_{z}
	= \iint_{R_{z} \times R_{z}} \psi(z_1,z_2) \phi(z_1,z_2) d\langle M\rangle_{z_1} d\langle N\rangle_{z_2} , ~~\forall \psi, \phi\in \cL^2_{MN}(z_0). 
\end{equation}

 {\bf Throughout the rest of this section, we only consider continuous strong martingales with finite fourth moments, unless otherwise stated.} 
 Then based on Theorem \ref{Thm-prop1.9}, we have 
\begin{equation}\label{e:quad1=quad2}
[M]^1=[M]^2=\langle M\rangle; ~~[N]^1=[N]^2=\langle N\rangle.
\end{equation}

To end this subsection, we shall follow the approach used in \cite[Section 6]{CW75} to show that $J_{MN}$ defined in \eqref{def-J} 
can be represented by $\psi\cdot MN$ with $\psi(z_1, z_2) =\1_{[z_1{\curlywedge\atop\curlywedge} z_2]}$. 

Recall the notations $z_{i,j} = (is_0/2^n,jt_0/2^n)$ and  $\rectangle_{i,j} = (z_{i,j},z_{i+1,j+1}]$.  We also denote $\epsilon_{i,j} 
= (z_{i0}, z_{i+1,j}]$ and $\delta_{i,j} = (z_{0,j}, z_{i,j+1}]$. Denote
\begin{align}
	J_{MN}^n(z)
	&:= \sum_{i,j=0}^{2^n-1} M( \delta_{i,j} \cap R_z ) N( \epsilon_{i,j} \cap R_z )\notag \\
	&= \sum_{i,j=0}^{2^n-1} \left( \sum_{k=0}^{i-1} M( (z_{k,j},z_{k+1,j+1}] \cap R_z ) \right) 
	\left( \sum_{l=0}^{j-1} N( (z_{i,l},z_{i+1,l+1}] \cap R_z ) \right)\notag \\
	&= \sum_{k<i} \sum_{l<j} \left(\psi_{kjil} \cdot MN\right)_z,\label{e:J^n}
\end{align}
where $\psi_{kjil}$ is given in \eqref{eq:psi}.
Thus, letting $n \rightarrow \infty$, we have
\begin{align*}
	J_{MN}^n(z) \rightarrow (\psi \cdot MN)_z,
\end{align*}
where 
\begin{align*}
	\psi(z_1,z_2) =\1_{[z_1{\curlywedge\atop\curlywedge} z_2]}=	\begin{cases}
	1, &\text{ if } z_1{\curlywedge\atop\curlywedge} z_2, \\
	0, & \mathrm{otherwise}.
	\end{cases}
\end{align*}

Define $M^n = \sum\limits_{i,j=0}^{2^n-1} \1_{\rectangle_{i,j}}(z) M_{z_{i,j}}$. Then $M^n$ is a  sequence of simple 
functions that approximate $M$ and 
\begin{align}
	\int_{R_{z_0}} M^n dN &= \sum_{i,j=0}^{2^n-1} M_{z_{i,j}} N(\rectangle_{i,j}) \nonumber \\
	&= \sum_{i,j=0}^{2^n-1} M_{z_{i,j}} \left( N(\epsilon_{i,j+1}) - N(\epsilon_{i,j}) \right) \nonumber \\
	&= \sum_{i,j=0}^{2^n-1} \left( M_{z_{i,j+1}} N(\epsilon_{i,j+1}) - M_{z_{i,j}} N(\epsilon_{i,j}) \right) + 
	\sum_{i,j=0}^{2^n-1} \left( M_{z_{i,j}} - M_{z_{i,j+1}} \right) N(\epsilon_{i,j+1}) \nonumber \\
	&= \sum_{i=0}^{2^n-1} M_{z_{i,2^n}} N(\epsilon_{i,2^n}) - \sum_{i,j=0}^{2^n-1} M(\delta_{i,j}) \left( N(\epsilon_{i,j}) 
	+ N(\rectangle_{i,j}) \right).\label{e:int-MdN}
\end{align}
If we define $\widetilde{M}_{s,t_0}^n = M_{is_0/2^n,t_0}$ for $s \in (is_0/2^n, (i+1)s_0/2^n]$, and $\delta^n(z) = M(\delta_{i,j})$ 
if $z \in \rectangle_{i,j}$. Let $H_{z_0}$ be the line segment with endpoints $(0,t_0)$ and $z_0 = (s_0,t_0)$, then
\begin{align}\label{e:0}
	\int_{R_{z_0}} M^n dN
	= \int_{H_{z_0}} \widetilde{M}_{s,t_0}^n(s)N(ds,t)
	- J_{MN}^n(z_0) - \int_{R_{z_0}} \delta^n dN.
\end{align}
By the Cauchy-Schwarz inequality,
\begin{align}
	\bE \left[ \bigg( \int_{R_{z_0}} \delta^n dN \bigg)^2 \right]
	&= \bE \left[ \int_{R_{z_0}} \left( \delta^n(s,t) \right)^2 d \langle N \rangle_{st} \right] \notag\\
	&\le \bE \left[ \int_{R_{z_0}} \sup_{i,j} M(\delta_{i,j})^2 d \langle N \rangle_{st} \right] \notag\\
	&\le \bE \left[ \sup_{i,j} M(\delta_{i,j})^2 \langle N \rangle_{z_0} \right] \notag\\
	&\le \left( \bE \left[ \sup_{i,j} M(\delta_{i,j})^4 \right] \bE \left[ \langle N \rangle_{z_0}^2 \right] \right)^{1/2}\notag \\
	&\to 0, ~ n\to \infty,\label{e:1}
\end{align}
where the last step holds due to the continuity of $M$ and the dominated convergence theorem, noting that  
$\bE[\sup_{i,j,n} M(\delta_{i,j})^4]$ is dominated by $\bE[\sup_{z\prec z_0} |M_z|^4]$, which is dominated by 
$(4/3)^8\bE[ |M_{z_0}|^4]$ due to Theorem \ref{Thm-BDG 2dim} and the existence of the fourth moment of the $M$.

Furthermore, Theorem \ref{Thm-BDG 2dim} yields 
\begin{align*}
	\bE \left[ \sup_n\sup_{z\prec z_0} (M_z^n-M_z)^4 \right]
	\le 8\bE \left[ \sup_n\sup_{z\prec z_0} |M_z^n|^4 + \sup_{z\prec z_0} |M_z|^4 \right]
	\le 16 \bE \left[ \sup_{z\prec z_0} |M_z|^4 \right]
	< \infty.
\end{align*}
By the Cauchy-Schwarz inequality, the dominated convergence theorem and the continuity of $M$, we have
\begin{align}
	\bE \left[ \bigg( \int_{R_{z_0}} (M^n-M) dN \bigg)^2 \right]
	&= \bE \left[ \int_{R_{z_0}} (M^n-M)^2 d \langle N \rangle \right] \notag\\
	&\le \bE \left[ \sup_{n,z} (M_z^n-M_z)^2 \langle N \rangle_{z_0} \right] \notag\\
	&\le \left( \bE \left[ \sup_{n,z} (M_z^n-M_z)^4 \right] \bE \left[ \langle N \rangle_{z_0}^2 \right] \right)^{1/2} \notag\\
	&\rightarrow 0, ~n \rightarrow \infty.\label{e:2}
\end{align}
Similarly, we can show the following $L^2$-convergence,
\begin{align}\label{e:3}
	\int_{H_{z_0}} \widetilde{M}_{s,t_0}^n N(ds,t)
	\rightarrow \int_{H_{z_0}} M(s,t) N(ds,t),~ n \rightarrow \infty.
\end{align}
Recalling that $\lim_{n\to\infty}J_{MN}^n(z) =(\psi \cdot MN)_z$ with $\psi(z_1,z_2) =\1_{[z_1{\curlywedge\atop\curlywedge} z_2]}$, 
by \eqref{e:0}, \eqref{e:1}, \eqref{e:2} and \eqref{e:3}, we have
\begin{align*}
	(\psi \cdot MN)_{z_0} = \int_{H_{z_0}} M(s,t) N(ds,t) - \int_{R_{z_0}} M dN, 
\end{align*}
and hence by the definition \eqref{def-J} of $J_{MN}$, we have
\begin{equation}\label{e:J}
J_{MN}(z_0)=(\psi \cdot MN)_{z_0}.
\end{equation}

Therefore, we can calculate $\langle J_{MN}\rangle$ by \eqref{e:qua-var}, 
\begin{align*}
\langle J_{MN}\rangle_z= \iint_{R_{z} \times R_{z}} \1_{[z_1{\curlywedge \atop \curlywedge} z_2]} d\langle M\rangle_{z_1} 
d\langle N\rangle_{z_2}=\int_{R_z} d_t\langle M\rangle_{st} d_s\langle N \rangle_{st},
\end{align*}
and hence 
\begin{equation}\label{e:quad-var-J}
d\langle J_{MN}\rangle_{st} = d_t\langle M\rangle_{st} d_s\langle N \rangle_{st}. 
\end{equation}

Furthermore, the following equality holds, 
\begin{align}\label{def-J'}
	J_{MN}(s_0,t_0)
	=& \int_0^{s_0} M(s,t_0) N(ds,t_0) - \int_{R_{s_0t_0}} M(s,t) dN(s,t)\notag\\
	=&  \int_0^{t_0} N(s_0,t) M(s_0,dt) - \int_{R_{s_0t_0}} N(s,t) dM(s,t).
\end{align}
This can be deduced by rewriting \eqref{e:int-MdN} as follows
\begin{align*}
	\int_{R_{z_0}} M^n dN
	&= \sum_{i,j=0}^{2^n-1} M_{z_{i,j}} N(\rectangle_{i,j}) \nonumber \\
	&= \sum_{i,j=0}^{2^n-1} M_{z_{i,j}} \left( N(\delta_{i+1,j}) - N(\delta_{i,j}) \right) \nonumber \\
	&= \sum_{i,j=0}^{2^n-1} \left( M_{z_{i+1,j}} N(\delta_{i+1,j}) - M_{z_{i,j}} N(\delta_{i,j}) \right) + 
	\sum_{i,j=0}^{2^n-1} \left( M_{z_{i,j}} - M_{z_{i+1,j}} \right) N(\delta_{i+1,j}) \nonumber \\
	&= \sum_{j=0}^{2^n-1} M_{z_{2^n, j}} N(\delta_{2^n,j}) - \sum_{i,j=0}^{2^n-1} M(\epsilon_{i,j}) \left( N(\delta_{i,j}) + N(\rectangle_{i,j}) \right).
\end{align*}
By letting $n$ go to infinity, we get for $\psi(z_1, z_2)=\1_{[z_1{\curlywedge \atop \curlywedge} z_2]}$,  
\[(\psi\cdot NM)_{z_0}= \int_0^{t_0} M(s_0, t) N(s_0, dt)-\int_{R_{z_0}} MdN.\]
This together with \eqref{e:J} implies \eqref{def-J'}.

\subsection{Multi-dimensional Green's formula for martingales on the plane}

Now we are ready to prove Theorem \ref{Thm-Green formula}, the multi-dimensional Green's formula  on the plane. 
Let $\{M^{(i)}(s,t), (s,t)\in\bR_+^2\}_{1 \le i \le d}$ be a family of independent continuous strong martingales on $\bR_+^2$ 
with finite fourth moment. We assume that the increasing process $\langle M^{(i)} \rangle$ is deterministic for  every$1\le i \le d$. 
Let $F_j = F_j(s,t), 1 \le j \le d$ be a sequence of predictable processes of the form, 
\begin{align} \label{eq-stochastic partial derivative}
	F_j(s,t) = F_j(s,0) + \sum_{i=1}^d \int_0^t f_{j,i}(s,r) M^{(i)}(s,dr) + \int_0^t f_{j,0} (s,r) dr,
\end{align}
where $f_{j,i}, 1\le j \le d, 0\le i\le d$ are $\cF$-predictable processes. 

\begin{theorem} \label{Thm-Green formula}
Fix $s_0, t_0 > 0$. Suppose that $\{F_j(s,t)\}_{1 \le j \le d}$ are predictable processes given by \eqref{eq-stochastic partial derivative}. 
Assume
\begin{align}\label{e:c1}
	\bE \left[ \int_0^{s_0} \int_0^{t_0} f_{j,i}(s,t)^2  d_t\big\langle M^{(i)} \big\rangle_{s_0 t} d_s\big\langle M^{(j)} \big\rangle_{st_0} \right] 
	< \infty, \qquad \forall 1 \le i, j \le d,
\end{align}
and
\begin{align}\label{e:c2}
	\bE \left[ \int_0^{t_0} \int_0^{s_0} f_{j,0}(s,t)^2 d_s \big\langle M^{(i)}\big\rangle_{st_0} dt \right] < \infty, \qquad \forall 1 \le j \le d.
\end{align}
%Assume that
%\begin{align*}
%	\bE \left[ \int_0^{s_0} F_j(s,0)^2 d_s\left\langle M^{(i)} \right\rangle_{st_0} \right] < \infty, \forall 1 \le i, j \le d.
%\end{align*}
%{\blue We may also assume that $f_{i,j} = f_{j,i}$. It depends on whether we can get the estimation on $dJ$. See the proof for reason.} 
Then for any rectangle $A \subseteq R_{s_0 t_0}$, we have
\begin{align} \label{eq-Thm1}
	\sum_{j=1}^d \int_{\partial A} F_j(s,t) M^{(j)}(ds,t) &= \sum_{j=1}^d \int_A F_j(s,t) dM^{(j)}(s,t) 
	+ \sum_{i,j=1}^d \int_A f_{j,i}(s,t) dJ_{M^{(i)},M^{(j)}}(s,t) \nonumber \\
	& \qquad  + \sum_{j=1}^d \int_A f_{j,0}(s,t) M^{(j)} (ds,t) dt.
\end{align}
\end{theorem}

\begin{proof} 
We will follow the argument in the proof of \cite[Theorem 6.1]{CW75}.  Let $A = [s_1, s_2] \times [t_1, t_2]\subset [0,s_0]\times[0,t_0]$. 
Without loss of generality, we may assume that $F_j = 0$ on the line segment with endpoints $(s_1,t_1)$ and $(s_2,t_1)$. 
Indeed, noting $F_j(s,t)= F_j(s, t_1) + (F_j(s,t)-F_j(s, t_1)),$  it follows from 
\begin{align*}
	\int_A F_j(s,t_1) dM^{(j)}(s,t)
	= \int_{s_1}^{s_2} F_j(s,t_1) \left( M^{(j)}(ds,t_2) - M^{(j)}(ds,t_1) \right)
	= \int_{\partial A} F_j(s,t_1) M^{(j)}(ds,t), 
\end{align*}
that \eqref{eq-Thm1} holds for $F_j(s,t)$ if and only if 
it holds for $F_j(s,t) - F_j(s,t_1)$.

Next, we consider the case that each stochastic partial derivative $f_{j,i}(s,t)\equiv f_{j,i}\in \cF_{s_1, t_1}$ is a constant function  
for $1 \le j \le d, 0 \le i \le d$. Then by \eqref{eq-stochastic partial derivative}, we have
\begin{align} \label{eq-2.18}
	F_j(s,t) = \sum_{i=1}^d f_{j,i} \left( M^{(i)}(s,t) - M^{(i)}(s,t_1) \right) + f_{j,0}(t-t_1), (s,t) \in A, 1 \le j \le d.
\end{align}
On one hand, noting that $J_{MN}(A)=J_{MN}(s_2,t_2)-J_{MN}(s_1,t_2)-J_{MN}(s_2,t_1)+J_{MN}(s_1,t_1)$, it follows from  \eqref{def-J} that
\begin{align} \label{eq-2.19}
	& \int_A f_{j,i}(s,t) dJ_{M^{(i)},M^{(j)}}(s,t) \nonumber \\
	&= \int_{\partial A} f_{j,i} M^{(i)}(s,t) M^{(j)}(ds,t) - \int_A f_{j,i} M^{(i)}(s,t) dM^{(j)}(s,t) \nonumber \\
	&=\int_{\partial A} f_{j,i} \left( M^{(i)}(s,t) - M^{(i)}(s,t_1) \right) M^{(j)}(ds,t) - \int_A f_{j,i} \left( M^{(i)}(s,t) - M^{(i)}(s,t_1) \right) dM^{(j)}(s,t).
\end{align}
Here $\int_{\partial A}$ is a line integral on $\partial A$ with clockwise as its positive direction.

On the other hand, It\^{o}'s formula yields
\begin{align} \label{eq-2.20}
	& \int_A f_{j,0} (t-t_1) dM^{(j)}(s,t) \nonumber \\
	&= f_{j,0} \int_{t_1}^{t_2} (t-t_1) \left( M^{(j)}(s_2,dt) - M^{(j)}(s_1,dt) \right) \nonumber \\
	&= f_{j,0} (t_2-t_1) \left( M^{(j)}(s_2,t_2) - M^{(j)}(s_1,t_2) \right) - f_{j,0} \int_{t_1}^{t_2} \left( M^{(j)}(s_2,t) - M^{(j)}(s_1,t) \right) dt \nonumber \\
	&= \int_{\partial A} f_{j,0} (t-t_1) M^{(j)}(ds,t) - \int_{t_1}^{t_2} \left( \int_{s_1}^{s_2} f_{j,0} M^{(j)}(ds,t) \right) dt.
\end{align}
By \eqref{eq-2.18}, \eqref{eq-2.19} and \eqref{eq-2.20}, we get \eqref{eq-Thm1}. Thus, we have proved the theorem 
for the case that all stochastic partial derivatives are constant functions. Note that for $A=\cup_{i=1}^k A_i$ where
 $A_i$ are disjoint rectangles, one has $\int_{\partial A}=\sum_{i=1}^k \int_{\partial A_i} $ and $\int_A=\sum_{i=1}^k \int_{A_i}$. 
 Therefore,  \eqref{eq-Thm1} also holds for the case that all stochastic partial derivatives are simple functions.

For the general case, recall that the martingales $\{M^{(i)}\}_{1 \le i \le d}$ are independent and the 
increasing processes $\{\left\langle M^{(i)} \right\rangle\}_{1 \le i \le d}$ are deterministic. By \eqref{e:c1} and \eqref{e:c2},   for $0 \le i \le d$, $1 \le j \le d$, 
we can find sequences $\{f_{j,i}^{(n)}\}_{n \in \bN}$ of bounded simple functions such that as $n\to \infty$,
\begin{align} \label{eq-condition on stochastic derivative 1}
	\int_0^{s_0} \int_0^{t_0} \bE \left[ \left( f_{j,i}^{(n)}(s,t) - f_{j,i}(s,t) \right)^2 \right] d_t\big\langle M^{(i)} \big\rangle_{s_0t}d_s 
	\big\langle M^{(j)} \big\rangle_{st_0} \rightarrow 0, ~1 \le i,j \le d,
\end{align}
and
\begin{align} \label{eq-condition on stochastic derivative 2}
	\int_0^{s_0} \int_0^{t_0} \bE \left[ \left( f_{j,0}^{(n)}(s,t) - f_{j,0}(s,t) \right)^2 \right] dt d_s \big\langle M^{(j)}\big\rangle_{st_0} 
	\rightarrow 0, \qquad 1 \le j \le d.
\end{align}
Define
\begin{align*}
	F_j^{(n)}(s,t) = \sum_{i=1}^d \int_0^t f_{j,i}^{(n)}(s,r) M^{(i)}(s,dr) + \int_0^t f_{j,0}^{(n)} (s,r) dr, ~1 \le j \le d.
\end{align*}
Then \eqref{eq-Thm1} holds for the family $\big\{ F_j^{(n)} \big\}_{1 \le j \le d}$, and it remains to take the limit as $n \rightarrow \infty$.

We deal with the left-hand side of \eqref{eq-Thm1} first. It follows from \eqref{eq-condition on stochastic derivative 1} that, 
as $n \rightarrow \infty$, for $1 \le i, j \le d$
\begin{align} \label{eq-2.23}
	& \bE \left[ \left( \int_{\partial A} \int_0^t f_{j,i}(s,r) M^{(i)}(s,dr) M^{(j)}(ds,t) - \int_{\partial A} \int_0^t f_{j,i}^{(n)}(s,r) M^{(i)}(s,dr) M^{(j)}(ds,t) \right)^2 \right] \nonumber \\
	&\le 2\sum_{k=1,2} \bE \left[ \left( \int_{s_1}^{s_2} \int_0^{t_k} \left( f_{j,i}(s,r) - f_{j,i}^{(n)}(s,r) \right) M^{(i)}(s,dr) M^{(j)}(ds,t_k) \right)^2 \right] \nonumber \\
	&= 2\sum_{k=1,2} \int_{s_1}^{s_2} \int_0^{t_k} \bE \left[ \left( f_{j,i}(s,r) - f_{j,i}^{(n)}(s,r) \right)^2 \right] d_r\big \langle M^{(i)} \big\rangle_{sr} d_s \big\langle M^{(j)} \big\rangle_{st_k} \nonumber \\
	&\le  2\sum_{k=1,2} \int_{s_1}^{s_2} \int_0^{t_k} \bE \left[ \left( f_{j,i}(s,r) - f_{j,i}^{(n)}(s,r) \right)^2 \right] d_r\big\langle M^{(i)} \big\rangle_{s_0r} d_s \big\langle M^{(j)} \big\rangle_{st_0} \nonumber \\
	&\rightarrow 0.
\end{align}

Similarly, by \eqref{eq-condition on stochastic derivative 2}, we have as $n\to\infty$, for $1\le  j\le d$,
\begin{align} \label{eq-2.24}
	&\bE \left[ \left( \int_{\partial A} \int_0^t f_{j,0}(s,r) dr M^{(j)}(ds,t) - \int_{\partial A} \int_0^t f_{j,0}^{(n)}(s,r) dr M^{(j)}(ds,t) \right)^2 \right] \nonumber \\
		&\le 2\sum_{k=1,2} \int_{s_1}^{s_2} \bE \left[ \left( \int_0^{t_k} \left( f_{j,0}(s,r) - f_{j,0}^{(n)}(s,r) \right) dr \right)^2 \right] d_s\big\langle M^{(j)} \big\rangle_{st_0} \nonumber \\
	&\le 4 \int_{s_1}^{s_2} t_0 \int_0^{t_0} \bE \left[ \left( f_{j,0}(s,r) - f_{j,0}^{(n)}(s,r) \right)^2 \right] dr d_s\big\langle M^{(j)} \big\rangle_{st_0} \nonumber \\
	&\rightarrow 0.
\end{align}
Hence, combing  \eqref{eq-2.23} with \eqref{eq-2.24}, we get 
\begin{align*}
	\lim_{n \rightarrow \infty} \sum_{j=1}^d \int_{\partial A} F_j^{(n)}(s,t) M^{(j)}(ds,t) = \sum_{j=1}^d \int_{\partial A} F_j(s,t) M^{(j)}(ds,t)
\end{align*}
in $L^2(\Omega)$.

Next, we deal with the first term on the right-hand side of \eqref{eq-Thm1}. By \eqref{eq-condition on stochastic derivative 1},
\begin{align} \label{eq-2.25}
	& \bE \left[ \left( \int_A \int_0^t f_{j,i}(s,r) M^{(i)}(s,dr) dM^{(j)}(s,t) - \int_A \int_0^t f_{j,i}^{(n)}(s,r) M^{(i)}(s,dr) dM^{(j)}(s,t) \right)^2 \right] \nonumber \\
	&= \int_A \bE \left[ \left( \int_0^t \left( f_{j,i}(s,r) - f_{j,i}^{(n)}(s,r) \right) M^{(i)}(s,dr) \right)^2 \right] d\big\langle M^{(j)} \big\rangle_{st} \nonumber \\
	&= \int_A \int_0^t \bE \left[ \left( f_{j,i}(s,r) - f_{j,i}^{(n)}(s,r) \right)^2 \right] d_r\big\langle M^{(i)}\big\rangle_{sr} d\big\langle M^{(j)} \big\rangle_{st} \nonumber \\
	&\le \int_A \int_0^{t_0} \bE \left[ \left( f_{j,i}(s,r) - f_{j,i}^{(n)}(s,r) \right)^2 \right] d_r\big\langle M^{(i)}\big\rangle_{sr} d\big \langle M^{(j)} \big\rangle_{st} \nonumber \\
	&\le \int_{0}^{s_0} \int_0^{t_0} \bE \left[ \left( f_{j,i}(s,r) - f_{j,i}^{(n)}(s,r) \right)^2 \right] d_r\big \langle M^{(i)}\big\rangle_{s_0r} d_s\big\langle M^{(j)} \big\rangle_{st_0} \nonumber \\
	&\rightarrow 0, \quad n \rightarrow \infty,\ \ \ \forall 1 \le i, j \le d.
\end{align}
Similarly, by \eqref{eq-condition on stochastic derivative 2},
\begin{align} \label{eq-2.26}
	&\bE \left[ \left( \int_A \int_0^t f_{j,0}(s,r) dr dM^{(j)}(s,t) - \int_A \int_0^t f_{j,0}^{(n)}(s,r) dr dM^{(j)}(s,t) \right)^2 \right] \nonumber \\
	&\le \int_A t_0 \int_0^{t_0} \bE \left[ \left( f_{j,0}(s,r) - f_{j,0}^{(n)}(s,r) \right)^2 \right] dr d\big \langle M^{(j)} \big\rangle_{st} \nonumber \\
	&\le t_0 \int_{0}^{s_0}  \int_0^{t_0} \bE \left[ \left( f_{j,0}(s,r) - f_{j,0}^{(n)}(s,r) \right)^2 \right] dr d_s \big\langle M^{(j)} \big\rangle_{st_0}\nonumber \\
	&\rightarrow 0, \quad n \rightarrow \infty,\ \forall 1 \le j \le d.
\end{align}
Hence, \eqref{eq-2.25} and  \eqref{eq-2.26} imply
\begin{align*}
	\lim_{n \rightarrow \infty} \sum_{j=1}^d \int_A F_j^{(n)}(s,t) M^{(j)}(ds,t) \rightarrow \sum_{j=1}^d \int_A F_j(s,t) M^{(j)}(ds,t)
\end{align*}
in $L^2(\Omega)$.

Next we deal with the limit of the second term on the right hand side of \eqref{eq-Thm1}. By \eqref{e:quad-var-J} and 
\eqref{eq-condition on stochastic derivative 1}, we have
\begin{align} \label{eq-2.27}
	 &\bE \left[ \left( \int_A f_{j,i}(s,t) dJ_{M^{(i)},M^{(j)}}(s,t) - \int_A f_{j,i}^{(n)}(s,t) dJ_{M^{(i)},M^{(j)}}(s,t) \right)^2 \right] \nonumber \\
	&\le  \bE \left[ \int_A \left| f_{j,i}(s,t) - f_{j,i}^{(n)}(s,t) \right|^2 d_t \big\langle M^{(i)} \big\rangle_{s_0t} d_s \big\langle M^{(j)}\big\rangle_{st_0} \right] \nonumber \\
	&\rightarrow 0, \quad n \rightarrow \infty, \ ~\forall 1 \le i, j \le d.
\end{align}
Hence, we have
\begin{align*}
	\sum_{i,j=1}^d \int_A f_{j,i}^{(n)}(s,t) dJ_{M^{(i)},M^{(j)}}(s,t) \rightarrow \sum_{i,j=1}^d \int_A f_{j,i}(s,t) dJ_{M^{(i)},M^{(j)}}(s,t)
\end{align*}
in $L^2(\Omega)$.

Lastly, we deal with the limit of the third term on the right hand side of \eqref{eq-Thm1}. By the Cauchy-Schwarz inequality 
and \eqref{eq-condition on stochastic derivative 2},
\begin{align} \label{eq-2.29}
	& \bE \left[ \left( \int_A f_{j,0}(s,t) M^{(j)} (ds,t) dt - \int_A f_{j,0}^{(n)}(s,t) M^{(j)} (ds,t) dt \right)^2 \right] \nonumber \\
	&= \bE \left[ \left( \int_{t_1}^{t_2} \int_{s_1}^{s_2} \left( f_{j,0}(s,t) - f_{j,0}^{(n)}(s,t) \right) M^{(j)} (ds,t) dt \right)^2 \right] \nonumber \\
	&\le \left( t_2-t_1 \right) \int_{T}^{t_2} \bE \left[ \left( \int_{S}^{s_2} \left( f_{j,0}(s,t) - f_{j,0}^{(n)}(s,t) \right) M^{(j)} (ds,t) \right)^2 \right] dt \nonumber \\
	&= \left( t_2-t_1 \right) \int_{T}^{t_2} \int_{S}^{s_2} \bE \left[ \left( f_{j,0}(s,t) - f_{j,0}^{(n)}(s,t) \right)^2 \right] \big\langle M^{(j)} (ds,t) \big\rangle dt \nonumber \\
	&\rightarrow 0, \quad n \rightarrow \infty,\  \forall 1 \le j \le d.
\end{align}
Thus, we have the following convergence in $L^2(\Omega)$,
\begin{align*}
	\sum_{j=1}^d \int_A f_{j,0}^{(n)}(s,t) M^{(j)} (ds,t) dt
	\rightarrow \sum_{j=1}^d \int_A f_{j,0}(s,t) M^{(j)} (ds,t) dt.
\end{align*}
 The proof is concluded.
\end{proof}

Similarly, for predictable processes of the form
\begin{align}\label{e:G}
	F_j(s,t) = F_j(0,t) + \sum_{i=1}^d \int_0^s f_{j,i}(r,t) M^{(i)}(dr,t) + \int_0^s f_{j,0} (r,t) dr, \ \ 1 \le j \le d,
\end{align}
where $f_{j,i}, \,1\le j\le d, \, 0 \le i \le d$ are $\cF$-predictable processes, 
we have the following Green's formula. 

\begin{theorem} \label{Coro-Green formula} Fix $s_0, t_0 > 0$. Suppose that $\{F_j(s,t)\}_{1\le j\le d}$ are predictable processes 
given by \eqref{e:G}. Assume 
\begin{align*}
	\bE \left[ \int_0^{s_0} \int_0^{t_0} f_{j,i}(s,t)^2 d_s\big\langle M^{(i)} \big\rangle_{st_0}d_t \big\langle M^{(j)}\big\rangle_{st_0} \right] 
	< \infty,\quad  \forall 1 \le i, j \le d,
\end{align*}
and
\begin{align*}
	\bE \left[ \int_0^{s_0} \int_0^{t_0} f_{j,0}(s,t)^2 d_t\big\langle M^{(i)} \big\rangle_{s_0 t} ds \right] < \infty, \quad  \forall 1 \le j \le d.
\end{align*}
Then for any rectangle $A \subseteq R_{s_0 t_0}$, we have
\begin{align*}
	\sum_{j=1}^d \int_{\partial A} F_j(s,t) M^{(j)}(s,dt) &= \sum_{j=1}^d \int_A F_j(s,t) dM^{(j)}(s,t) + 
	\sum_{i,j=1}^d \int_A f_{j,i}(s,t) dJ_{M^{(j)},M^{(i)}}(s,t)\nonumber \\
	&\qquad \qquad + \sum_{j=1}^d \int_A f_{j,0}(s,t) M^{(j)} (s,dt) ds.
\end{align*}
\end{theorem}

\begin{proof}
Noting that by the second equality of  \eqref{def-J'}, we have that for the left-hand side of \eqref{eq-2.19}, 
\begin{align*} 
	 \int_A f_{j,i}(s,t) dJ_{M^{(j)},M^{(i)}}(s,t)% \nonumber \\
	&= \int_{\partial A} f_{j,i} M^{(i)}(s,t) M^{(j)}(s,dt) - \int_A f_{j,i} M^{(i)}(s,t) dM^{(j)}(s,t).
\end{align*}
The rest of the proof is the same as that of Theorem \ref{Thm-Green formula} and thus is omitted.
\end{proof}

\subsection{Quadratic covariations of $J_{MN}$ and $J_{M'N'}$ }

Let $M,N, M', N'$ be continuous martingales belonging to $\mathfrak M_s^4(z_0).$  In this subsection, for the completion of the theory, we shall derive the quadratic covariation for $J_{MN} = \psi \cdot MN$ and $J_{M'N'} = \psi \cdot M'N'$ 
with $\psi (z_1, z_2) = \1_{[z_1 {\curlywedge\atop \curlywedge} z_2]}$  which are defined in Section \ref{sec:J-MN}. 
More specifically, we aim to show
 \begin{equation}\label{e:quad-var-J'}
d\langle J_{MN}, J_{M',N'} \rangle_{st} = d_t\langle M, M'\rangle_{st} d_s\langle N, N' \rangle_{st}.
\end{equation}

Recall  that $J_{MN}$ can be approximated by $J^n_{MN}$ as in \eqref{e:J^n}, and that  one can approximate the 
function $\psi (z_1, z_2) = \1_{[z_1 {\curlywedge\atop \curlywedge} z_2]}$ by 
\begin{align*}
	\psi = \lim_{n\to \infty} \sum_{i,j, k,l\in \mathbf I_n} \psi_{ijkl},
\end{align*}
where $\psi_{ijkl}(z_1, z_2) = \1_{\rectangle_{i,j}}(z_1) \1_{\rectangle_{k,l}}(z_2)$ and  $\mathbf I_n$ is a subset of 
$\{(i,j, k, l), i, j, k,l \in 1, \dots, 2^n\}$ which  consists of $(i,j,k,l)$ satisfying $0\le i<k\le 2^n-1$ and $0\le l<j\le 2^n-1$.  Denote by $\mathbf J_n$  the subset of $\mathbf I_n \times \mathbf{I}_n$ such that for $((i,j,k,l),(i',j',k',l')) \in \mathbf J_n$, the four rectangles $A=\rectangle_{i,j}$, $B=\rectangle_{k,l}$, $A'=\rectangle_{i',j'}$, $B'=\rectangle_{k',l'}$ are of the same position as in  Case 3 in the proof of \eqref{e:increment} in Section \ref{sec:J-MN}. That is, $A$ and $A'$ are at the same horizontal level and are at the upper left of $B$ and $B'$, while $B$ and $B'$ are at the same vertical level. Now  the quadratic covariation can be computed as follows,
\begin{align} \label{eq-quadratic JMN and JM'N'}
	\left\langle J_{MN}, J_{M'N'} \right\rangle_{z_0}  
	&= \lim_{n\to \infty} \sum_{(i,j,k,l) \in \mathbf I_n} \sum_{(i',j',k',l') \in \mathbf  I_n} \left\langle \psi_{ijkl} \cdot MN, \psi_{i'j'k'l'} \cdot M'N' \right\rangle_{z_0} \nonumber \\
	&= \lim_{n\to \infty} \Bigg( \sum_{((i,j,k,l), (i',j',k',l')) \in \mathbf J_n} \left\langle \psi_{ijkl} \cdot MN, \psi_{i'j'k'l'} \cdot M'N' \right\rangle_{z_0} \notag\\
	&\qquad \qquad\qquad  + \sum_{((i,j,k,l), (i',j',k',l')) \notin \mathbf J_n} \left\langle \psi_{ijkl} \cdot MN, \psi_{i'j'k'l'} \cdot M'N' \right\rangle_{z_0} \Bigg).
\end{align}
For the first term on the right-hand side of \eqref{eq-quadratic JMN and JM'N'}, observing that the indices
$((i,j,k',l'), (i',j',k,l)),$ $ ((i',j',k,l), (i,j,k',l'))$, and $((i',j',k',l'), (i,j,k,l))$ all belong to  $\mathbf J_n$ as long as $((i,j,k,l), (i',j',k',l'))$ 
$ \in\mathbf  J_n$. Thus, we have
\begin{align} \label{eq-quadratic symmetric}
	&\sum_{((i,j,k,l), (i',j',k',l')) \in \mathbf J_n} \left\langle \psi_{ijkl} \cdot MN, \psi_{i'j'k'l'} \cdot M'N' \right\rangle_{z_0}\nonumber \\
	&= \dfrac{1}{4} \sum_{((i,j,k,l), (i',j',k',l')) \in \mathbf J_n} \Big( \left\langle \psi_{ijkl} \cdot MN, \psi_{i'j'k'l'} \cdot M'N' \right\rangle_{z_0}
	+ \left\langle \psi_{ijk'l'} \cdot MN, \psi_{i'j'kl} \cdot M'N' \right\rangle_{z_0} \nonumber \\
	&\qquad \qquad \qquad \qquad+ \left\langle \psi_{i'j'kl} \cdot MN, \psi_{ijk'l'} \cdot M'N' \right\rangle_{z_0}
	+ \left\langle \psi_{i'j'k'l'} \cdot MN, \psi_{ijkl} \cdot M'N' \right\rangle_{z_0} \Big) \nonumber \\
	&= \sum_{((i,j,k,l), (i',j',k',l')) \in\mathbf J_n} \langle M,M'\rangle(\rectangle_{i,j} \cap \rectangle_{i',j'}) 
	\langle N,N'\rangle (\rectangle_{k,l} \cap \rectangle_{k',l'})\nonumber\\
	&= \sum_{(i,j,k,l)\in \mathbf  I_n} \langle M,M'\rangle (\rectangle_{i,j}) \langle N,N'\rangle (\rectangle_{k,l}),
\end{align}
where the second equality follows from \eqref{e:increment}.

For the second term in \eqref{eq-quadratic JMN and JM'N'}, noting that when $((i,j,k,l), (i',j',k',l')) \notin \mathbf J_n$, the four 
rectangles $A=\rectangle_{i,j}$, $B=\rectangle_{k,l}$, $A'=\rectangle_{i',j'}$ and $B'=\rectangle_{k',l'}$ are of the same position 
as in Case 1 or Case 2 in the proof of \eqref{e:increment} in Section \ref{sec:J-MN}. Thus, we have
\begin{align} \label{eq-quadratic nonsymmetric}
	\sum_{((i,j,k,l), (i',j',k',l')) \notin \mathbf J_n} \left\langle \psi_{ijkl} \cdot MN, \psi_{i'j'k'l'} \cdot M'N' \right\rangle_{z_0}= 0.
\end{align}
Therefore, substituting \eqref{eq-quadratic symmetric} and \eqref{eq-quadratic nonsymmetric} into \eqref{eq-quadratic JMN and JM'N'}, 
one has
\begin{align*}
	  \left\langle J_{MN}, J_{M'N'} \right\rangle_{z_0}  
	&= \lim_{n\to \infty} \sum_{(i,j,k,l)\in \mathbf  I_n}\langle M,M'\rangle (\rectangle_{ij}) \langle N,N'\rangle(\rectangle_{kl}) \\
	&= \int_{R_{z_0}}\int_{R_{z_0}} \1_{[z_1{\curlywedge \atop \curlywedge }z_2]} d\langle M,M'\rangle_{z_1}d\langle N,N'\rangle_{z_2},
\end{align*}
and this implies \eqref{e:quad-var-J'}.

\begin{remark} One can easily check that the computation is still valid if the function $\psi(z_1, z_2)$ is the limit of 
$\psi^{(n)}(z_1,z_2)$ in $\mathcal L_{MN}^2(z_0) $ and in $\mathcal L_{M'N'}^2(z_0)$ satisfying
\begin{align} \label{eq-condition on psi}
	\psi^{(n)}(z_1,z_2) \psi^{(n)}(z_1',z_2') = \psi^{(n)}(z_1,z_2') \psi^{(n)}(z_1',z_2),
\end{align} 
for all $z_1=(s_1, t_1), z_2=(s_2, t_2), z_1'=(s_1',t_1'), z_2'=(s_2', t_2')$ satisfying $\max \{s_1, s_2\} \le s_1' = s_2'$ and 
$\max \{t_1', t_2'\} \le t_1 = t_2$. Clearly, $\psi^{(n)}(z_1,z_2) = h_1(z_1)h_2(z_2)$ satisfies \eqref{eq-condition on psi}.  
Moreover, by fixing $(z_1', z_2')$, one can check that all the functions satisfying \eqref{eq-condition on psi} must have the 
form $\psi^{(n)}(z_1,z_2) = h_1(z_1)h_2(z_2)$. In this situation, we have 
 \begin{equation}\label{e:qua-var2}
\langle \psi\cdot MN, \psi\cdot M'N' \rangle_{z}
	= \iint_{R_{z} \times R_{z}} |\psi(z_1,z_2)|^2 d\langle M, M'\rangle_{z_1} d\langle N,N'\rangle_{z_2} . 
\end{equation}
\end{remark}

\section{SPDEs for the eigenvalue processes}\label{sec:SPDE}

In this section, we will derive a system of SPDEs satisfied by the eigenvalue processes of the Brownian sheet 
matrix $\X$ defined in \eqref{e:X}. We assume that the family $b(s,t)$ of independent Brownian sheets have deterministic initial values such that  the  eigenvalues of the  symmetric matrix $\X(0, 0)$ are distinct.

Recall that the standard $1$-dimensional Brownian sheet $\{B(s,t), (s,t)\in\bR_+^2\}$ is a centered Gaussian random 
field with covariance function
\begin{align*}
	\bE \left[ B(s_1,t_1) B(s_2,t_2) \right] = (s_1 \wedge s_2)  (t_1 \wedge t_2).
\end{align*}
It follows directly from L\'evy's characterization of Brownian motion that for any fixed $t_1,\ t_2 > 0$,
\begin{align*}
	\dfrac{1}{\sqrt{t_1}} B(t_1, \cdot), \
	\dfrac{1}{\sqrt{t_2}} B(\cdot, t_2)
\end{align*}
are standard $1$-dimensional Brownian motions.

Consider the Brownian sheet matrix defined in \eqref{e:X}.  As in Appendix \ref{sec:appendix}, for $1\le i \le d$,  
let $\lambda_i (s,t) := \Phi_i(b (s,t)) =\tilde \Phi_i(\X(s,t))$ be the $i$-th biggest eigenvalue of $\X(s,t)$, where the 
function $\tilde \Phi_i: \mathbf S_d\to \bR$ maps a $d\times d$ symmetric matrix $A\in\mathbf S_d$ to its 
$i$-th biggest eigenvalue $\tilde \Phi_i(A)$.

Let $S, T > 0$ be constants. By applying  It\^{o}'s formula to $\lambda_i(S, \cdot)$, we have
\begin{align} \label{eq-3.1}
	&\lambda_i (S,T) - \lambda_i(0,0)= \lambda_i (S,T) - \lambda_i(S,0) 
	= \Phi_i(b (S, T)) - \Phi_i(b (S, 0)) \nonumber \\
	&= \sum_{k \le h} \int_0^{T} \dfrac{\partial \Phi_i}{\partial b_{kh}}(b(S,t)) b_{kh}(S,dt) 
	+ \dfrac{1}{2} \sum_{k \le h} \int_0^{T} \dfrac{\partial^2 \Phi_i}{\partial b_{kh}^2}(b(S,t)) \langle b_{kh}(S,dt) \rangle \nonumber \\
	&= \sum_{k \le h} \int_0^{T} \dfrac{\partial \Phi_i}{\partial b_{kh}}(b(S,t)) b_{kh}(S,dt) 
	+ \dfrac{S}{2} \sum_{k \le h} \int_0^{T} \dfrac{\partial^2 \Phi_i}{\partial b_{kh}^2}(b(S,t)) dt.
\end{align}

By \eqref{eq-3.1} and \eqref{eq-2.2}, we have
\begin{align} \label{eq-3.2}
	\lambda_i (S,T) - \lambda_i(0,0)
	= \sum_{k \le h} \int_0^{T} \dfrac{\partial \Phi_i}{\partial b_{kh}}(b(S,t)) b_{kh}(S,dt) + 
	\sum_{j:j \neq i}\int_0^{T}  \dfrac{S}{\lambda_i (S,t) - \lambda_j (S,t)} dt.
\end{align}
We shall express the right-hand side of \eqref{eq-3.2} as a sum of double integrals on $[0,S]\times[0,T]$. 
We first deal with the second term. 

For $i \neq j$, as in Appendix \ref{sec:appendix} we denote for any $x \in \bR^{d(d+1)/2}$,
\begin{equation}\label{e:psi}
\Psi_{ij}(x)=\frac{1}{\Phi_i(x)-\Phi_j(x)}.
\end{equation}
By It\^{o}'s formula, we have
\begin{align} \label{eq-3.4}
	& \dfrac{S}{\lambda_i (S,t) - \lambda_j (S,t)} = S \Psi_{ij}(b(S,t))\nonumber \\
	&= \int_0^{S} \dfrac{1}{\lambda_i(s,t) - \lambda_j(s,t)} ds
	+ \sum_{k \le h} \int_0^{S} s \dfrac{\partial \Psi_{ij}}{\partial b_{kh}}(b(s,t)) b_{kh}(ds,t) \nonumber \\
	&\qquad\quad + \dfrac{1}{2} \sum_{k \le h} \int_0^{S} s \dfrac{\partial^2 \Psi_{ij}}{\partial b_{kh}^2}(b(s,t)) 
	\langle b_{kh}(ds,t) \rangle\nonumber \\
	&= \int_0^{S} \dfrac{1}{\lambda_i(s,t) - \lambda_j(s,t)} ds
	+ \sum_{k \le h} \int_0^{S} s \dfrac{\partial \Psi_{ij}}{\partial b_{kh}}(b(s,t)) b_{kh}(ds,t) \nonumber \\
	&\qquad \quad + \dfrac{1}{2} \sum_{k \le h} \int_0^{S} st \dfrac{\partial^2\Psi_{ij}}{\partial b_{kh}^2}(b(s,t)) ds.
\end{align}
Substituting \eqref{eq-2.11} into \eqref{eq-3.4}, we have
\begin{align} \label{eq-3.5}
	&\dfrac{S}{\lambda_i (S,t) - \lambda_j (S,t)} \nonumber \\
	&= \int_0^{S} \dfrac{1}{\lambda_i(s,t) - \lambda_j(s,t)} ds + \sum_{k \le h} \int_0^{S} s 
	\dfrac{\partial \Psi_{ij}}{\partial b_{kh}}(b(s,t)) b_{kh}(ds,t) \nonumber \\
	&\qquad  +	\int_0^{S} \dfrac{2st}{(\lambda_i (s,t) - \lambda_j (s,t))^3} ds \nonumber \\
	&\qquad   + \int_0^{S} \dfrac{1}{(\lambda_i (s,t) - \lambda_j (s,t))} \sum_{l:l \neq i,j} \dfrac{st}{\left( \lambda_i (s,t) - \lambda_l (s,t) \right) 
	\left( \lambda_j (s,t) - \lambda_l (s,t) \right)} ds.
\end{align}

Lastly, we substitute \eqref{eq-3.5} to \eqref{eq-3.2} to obtain
\begin{align} \label{eq-3.6}
	& \lambda_i (S,T) - \lambda_i(0,0) \nonumber \\
	&= \sum_{k \le h} \int_0^{T} \dfrac{\partial \Phi_i}{\partial b_{kh}}(b(S,t)) b_{kh}(S,dt) 
	+ \sum_{j:j \neq i}\int_0^{T} \int_0^{S}  \dfrac{1}{\lambda_i(s,t) - \lambda_j(s,t)} ds dt \nonumber \\
	&\qquad + \sum_{j:j \neq i} \sum_{k \le h}  \int_0^{T} \int_0^{S} s \dfrac{\partial \Psi_{ij}}{\partial b_{kh}}(b(s,t))  b_{kh}(ds,t) dt \nonumber \\
	&\qquad +  \sum_{j:j \neq i} \int_0^{T} \int_0^{S} \dfrac{2st}{(\lambda_i (s,t) - \lambda_j (s,t))^3} ds dt \nonumber \\
	&\qquad +  \sum_{j\neq l: j \neq i, l\neq i} \int_0^{T}  \int_0^{S} \dfrac{1}{(\lambda_i (s,t) - \lambda_j (s,t))}  
	\dfrac{st}{\left( \lambda_i (s,t) - \lambda_l (s,t) \right) \left( \lambda_j (s,t) - \lambda_l (s,t) \right)} dsdt.
\end{align}
The last term on the right-hand side of \eqref{eq-3.6} vanishes, noting that it sums  over all $j\neq l$ for $j,l \neq i$ and that  
$ \frac{1}{(\lambda_i  - \lambda_j)}  	\frac{st}{\left( \lambda_i - \lambda_l \right) \left( \lambda_j - \lambda_l \right)}$ changes 
its sign by interchanging the indices $j$ and $l$. Therefore, we have
\begin{align} \label{eq:pde}
	& \lambda_i (S,T) - \lambda_i(0,0) \nonumber \\
	&= \sum_{k \le h} \int_0^{T} \dfrac{\partial \Phi_i}{\partial b_{kh}}(b(S,t)) b_{kh}(S,dt) 
	+ \sum_{j:j \neq i} \int_0^{T} \int_0^{S}  \dfrac{1}{\lambda_i(s,t) - \lambda_j(s,t)} ds dt \nonumber \\
	&\qquad + \sum_{j:j \neq i} \sum_{k \le h}\int_0^{T}  \int_0^{S} s \dfrac{\partial \Psi_{ij}}{\partial b_{kh}}(b(s,t))  b_{kh}(ds,t) dt \nonumber \\
	&\qquad + \sum_{j:j \neq i} \int_0^{T} \int_0^{S} \dfrac{2st}{(\lambda_i (s,t) - \lambda_j (s,t))^3} dsdt.
\end{align}

Now, we apply the multi-dimensional  Green's formula (Theorem \ref{Coro-Green formula}) to the first term on the right-hand 
side of  \eqref{eq:pde}. By \cite[Theorem 2.1]{Jaramillo2018}   (see also \cite[Theorem 1.1]{SXY20}), it has positive probability 
for the eigenvalues $\{\lambda_i(s,t), 1\le i\le d\}$ of the Brownian sheet matrix  $\X$ to collide. To avoid the singularity at the 
collisions,  we shall restrict $(s,t)$ in a region where all eigenvalues keep a distance from each other.

Define the region $D_\epsilon$ for  $\epsilon>0$  by 
\begin{align*}
	D_{\epsilon} = \left\{ (x_1, \ldots, x_d) \in \bR^d: x_i - x_{i+1} > \epsilon,  1 \le i \le d-1 \right\}. 
\end{align*}
 Let $\chi_{\epsilon} \in C_b^{\infty}(\bR^d)$ such that $\chi_{\epsilon}(x) = 1$ for $x \in D_{\epsilon}$ and 
 $\chi_{\epsilon}(x) = 0$ for $x \in \bR^d \setminus D_{\frac\epsilon2}$. For simplicity, we denote 
 $\Phi = (\Phi_1, \ldots, \Phi_d)$. By It\^{o}'s formula, we have
\begin{align} \label{eq-3.69}
	&\left( \dfrac{\partial \Phi_i}{\partial b_{kh}} \chi_{\epsilon} (\Phi) \right) (b(s,t)) \nonumber \\
	&= \left( \dfrac{\partial \Phi_i}{\partial b_{kh}} \chi_{\epsilon} (\Phi) \right) (b(0,t)) \nonumber \\
	&\qquad + \sum_{k' \le h'} \int_0^{s} \left( \dfrac{\partial^2 \Phi_i}{\partial b_{kh} \partial b_{k'h'}} \chi_{\epsilon} (\Phi) 
	+ \dfrac{\partial \Phi_i}{\partial b_{kh}} \sum_{l=1}^d \dfrac{\partial \chi_{\epsilon}}{\partial x_l} (\Phi) 
	\dfrac{\partial \Phi_l}{\partial b_{k'h'}} \right)(b(r,t)) b_{k'h'}(dr,t) \nonumber \\
	&\qquad + \dfrac{t}{2} \sum_{k' \le h'} \int_0^s \left( \dfrac{\partial^3 \Phi_i}{\partial b_{kh} \partial b_{k'h'}^2} \chi_{\epsilon} (\Phi) 
	+ 2 \dfrac{\partial^2 \Phi_i}{\partial b_{kh} \partial b_{k'h'}} \sum_{l=1}^d \dfrac{\partial \chi_{\epsilon}}{\partial x_l} (\Phi) 
	\dfrac{\partial \Phi_l}{\partial b_{k'h'}} \right. \nonumber \\
	&\qquad \qquad  \left. + \dfrac{\partial \Phi_i}{\partial b_{kh}} \sum_{l,l'=1}^d \dfrac{\partial^2 \chi_{\epsilon}}{\partial x_l \partial x_{l'}} (\Phi) 
	\dfrac{\partial \Phi_l}{\partial b_{k'h'}} \dfrac{\partial \Phi_{l'}}{\partial b_{k'h'}} + \dfrac{\partial \Phi_i}{\partial b_{kh}} 
	\sum_{l=1}^d \dfrac{\partial \chi_{\epsilon}}{\partial x_l} (\Phi) \dfrac{\partial^2 \Phi_l}{\partial b_{k'h'}^2} \right) (b(r,t)) dr.
\end{align}

Note that the function $\chi_{\epsilon}$ and all its partial derivatives vanish when $x \in \bR^d \setminus D_{\frac\epsilon2}$, 
by Lemma \ref{Lemma-eigen derivative}, all the integrand functions in \eqref{eq-3.69} are bounded. Hence, we can apply 
 Theorem  \ref{Coro-Green formula} to obtain
\begin{align}
	& \sum_{k \le h} \int_0^{T} \left( \dfrac{\partial \Phi_i}{\partial b_{kh}} \chi_{\epsilon} (\Phi) \right) (b(S,t)) b_{kh}(S,dt) \nonumber \\
	&= \sum_{k \le h} \int_{\partial R_{ST}} \left( \dfrac{\partial \Phi_i}{\partial b_{kh}} \chi_{\epsilon} (\Phi) \right)(b(s,t)) b_{kh}(s,dt) \nonumber \\
	&= \sum_{k \le h} \iint_{R_{ST}} \left( \dfrac{\partial \Phi_i}{\partial b_{kh}} \chi_{\epsilon} (\Phi) \right)(b(s,t)) db_{kh}(s,t) \nonumber \\
	&\qquad + \sum_{k \le h} \sum_{k' \le h'} \iint_{R_{ST}} \left( \dfrac{\partial^2 \Phi_i}{\partial b_{kh} \partial b_{k'h'}} \chi_{\epsilon} (\Phi) 
	+ \dfrac{\partial \Phi_i}{\partial b_{kh}} \sum_{l=1}^d \dfrac{\partial \chi_{\epsilon}}{\partial x_l} (\Phi) \dfrac{\partial \Phi_l}
	{\partial b_{k'h'}} \right) (b(s,t)) dJ_{ b_{kh} b_{k'h'}}(s,t) \nonumber \\
	&\qquad + \sum_{k \le h} \iint_{R_{ST}} \dfrac{t}{2} \sum_{k' \le h'} \Bigg( \dfrac{\partial^3 \Phi_i}{\partial b_{kh} \partial b_{k'h'}^2} \chi_{\epsilon} (\Phi) + 2 \dfrac{\partial^2 \Phi_i}{\partial b_{kh} \partial b_{k'h'}} \sum_{l=1}^d \dfrac{\partial \chi_{\epsilon}}{\partial x_l} 
	(\Phi) \dfrac{\partial \Phi_l}{\partial b_{k'h'}}  \nonumber \\
	&\qquad + \dfrac{\partial \Phi_i}{\partial b_{kh}} \sum_{l,l'=1}^d \dfrac{\partial^2 \chi_{\epsilon}}{\partial x_l \partial x_{l'}} (\Phi) 
	\dfrac{\partial \Phi_l}{\partial b_{k'h'}} \dfrac{\partial \Phi_{l'}}{\partial b_{k'h'}} + \dfrac{\partial \Phi_i}{\partial b_{kh}} 
	\sum_{l=1}^d \dfrac{\partial \chi_{\epsilon}}{\partial x_l} (\Phi) \dfrac{\partial^2 \Phi_l}{\partial b_{k'h'}^2} \Bigg) (b(s,t)) b_{kh}(s,dt) ds. \label{e:3.9}
\end{align}

Denote 
\begin{align*}
	T_{\epsilon} = \left\{ (s,t): (\Phi_1(b (s, t)), \ldots, \Phi_d(b (s,t))) \notin D_{\epsilon} \right\}.
\end{align*}
We shall construct a sequence of adapted random time pairs $\{(\sigma_{\frac1n}, \tau_{\frac1n})\}_{n\ge1}$ such that 
$(\sigma_{\frac1n}, \tau_{\frac1n})\prec (\sigma_{\frac1{n+1}}, \tau_{\frac1{n+1}})$. First,  we choose a pair of random 
times $(\sigma_1, \tau_1)$ as follows.  For each fixed $\omega\in \Omega$, if $T_{1}(\omega) = \emptyset$, then we 
choose $\sigma_1(\omega) = \tau_1(\omega) = \infty$; if $T_1(\omega) \not= \emptyset$, then by Zorn's lemma, there 
exists a minimal element $(s_1, t_1)$ in $T_{1}(\omega)$, and we set $(\sigma_1(\omega), \tau_1(\omega)) = (s_1, t_1).$ 
By the meaning of minimal element, we have $[(s,t)\prec\prec (\sigma_1, \tau_1)] = [(\Phi_1(b (s, t)), \ldots, \Phi_d(b (s,t))) 
\in D_{1} ] \in \cF_{st}$. Next, for an arbitrary fixed $\omega\in \Omega$, let $(\sigma_{\frac12}, \tau_{\frac12})$ be a minimal 
element of the set
\[\left\{ (\sigma_1(\omega),\tau_1(\omega))\prec (s,t): (\Phi_1(b (s, t)), \ldots, \Phi_d(b (s,t))) \notin D_{\frac12} \right\},\]
and  $ (\sigma_{\frac12}, \tau_{\frac12}) = (\infty, \infty)$ if the set is empty.
Clearly $(\sigma_{1}, \tau_1)\prec (\sigma_{\frac12},\tau_{\frac12})$, 
\[[(\sigma_1, \tau_1)\prec (s,t)\prec\prec(\sigma_{\frac12}, \tau_{\frac12})]=[(\Phi_1(b (s, t)), \ldots, \Phi_d(b (s,t))) 
\in D_{\frac12}\backslash D_{1} ] \in \cF_{st},\]
and hence $[ (s,t)\prec\prec(\sigma_{\frac12}, \tau_{\frac12})]\in \cF_{st}$.   The rest of random time pairs 
$(\sigma_{\frac1n},\tau_{\frac1n})$ can be constructed in the same way. Define
\begin{equation}\label{e:sigma-tau}
(\sigma,\tau)=\sup_{n\ge 1} (\sigma_{\frac1n}, \tau_{\frac1n}).
\end{equation} 
Thus, $[ (s,t)\prec\prec(\sigma, \tau)]=\cup_{n\ge1}[ (s,t)\prec\prec(\sigma_{\frac1n},\tau_{\frac1n})]\in \cF_{st}$.

 For each $n\ge 1$, on the set  $[\omega\in\Omega: (S,T)\prec\prec (\sigma_{\frac1n}(\omega), \tau_{\frac1n}(\omega))]$, 
 we have  that for $(s,t)\prec (S,T)$, $\Phi(b(s,t)) = (\Phi_1(b(s,t), \ldots, \Phi_d(b(s,t))$ belongs to  $D_{\frac1n}$ and all 
 the partial derivatives of the function $\chi_{\frac1n}$ vanish. 
Thus, by \eqref{e:3.9}, we have for $(S,T)\prec\prec (\sigma, \tau)$, 
\begin{align} \label{eq-3.11}
	&\sum_{k \le h} \int_0^{T} \dfrac{\partial \Phi_i}{\partial b_{kh}}(b(S,t)) b_{kh}(S,dt)
	= \sum_{k \le h} \int_{\partial R_{ST}} \dfrac{\partial \Phi_i}{\partial b_{kh}}(b(s,t)) b_{kh}(s,dt) \nonumber \\
	&= \sum_{k \le h} \iint_{R_{ST}} \dfrac{\partial \Phi_i}{\partial b_{kh}}(b(s,t)) db_{kh}(s,t)+ \sum_{k \le h} 
	\sum_{k' \le h'} \iint_{R_{ST}} \dfrac{\partial^2 \Phi_i}{\partial b_{kh} \partial b_{k'h'}}(b(s,t)) dJ_{ b_{kh} b_{k'h'}}(s,t) \nonumber \\
	&\quad + \sum_{k \le h} \iint_{R_{ST}} \dfrac{t}{2} \sum_{k' \le h'} \dfrac{\partial^3 \Phi_i}{\partial b_{kh} 
	\partial b_{k'h'}^2}(b(s,t)) b_{kh}(s,dt) ds.
\end{align}

Therefore, substitute \eqref{eq-3.11} to \eqref{eq:pde}, we have for $(S, T)\prec\prec (\sigma, \tau)$ and $1\le i\le d$, 
\begin{align} \label{eq-spde}
	&\lambda_i (S,T) - \lambda_i(0,0)\notag\\
	& = \sum_{k \le h} \iint_{R_{ST}} \dfrac{\partial \Phi_i}{\partial b_{kh}}(b(s,t)) db_{kh}(s,t)+ \sum_{k \le h} \sum_{k' \le h'} 
	\iint_{R_{ST}} \dfrac{\partial^2 \Phi_i}{\partial b_{kh} \partial b_{k'h'}}(b(s,t)) dJ_{ b_{kh} b_{k'h'}}(s,t) \nonumber \\
	&\qquad+ \sum_{k \le h} \sum_{k' \le h'}\iint_{R_{ST}} \dfrac{t}{2}  \dfrac{\partial^3 \Phi_i}{\partial b_{kh}
	\partial b_{k'h'}^2}(b(s,t)) b_{kh}(s,dt) ds + \sum_{j:j \neq i} \int_0^{T} \int_0^{S} \dfrac{1}{\lambda_i(s,t) - \lambda_j(s,t)} ds dt \nonumber \\
	&\qquad + \sum_{j:j \neq i} \sum_{k \le h}\int_0^{T}  \int_0^{S} \dfrac{\partial \Psi_{ij}}{\partial b_{kh}}(b(s,t)) s b_{kh}(ds,t) dt  
	+  \sum_{j:j \neq i} \int_0^{T}\int_0^{S} \dfrac{2st}{(\lambda_i (s,t) - \lambda_j (s,t))^3} ds dt.
\end{align}

Noting that by \eqref{e:psi},
\begin{align*}
	\dfrac{\partial \Psi_{ij}}{\partial b_{kh}} (b(s,t))
	&= \dfrac{-1}{(\lambda_i(s,t) - \lambda_j(s,t))^2} \left( \dfrac{\partial \Phi_i}{\partial b_{kh}} - \dfrac{\partial \Phi_j}{\partial b_{kh}} \right), 
%	&= \dfrac{-\left( \1_{[k<h]} + \sqrt{2} \1_{[k=h]} \right)}{(\lambda_i - \lambda_j)^2} \left( \dfrac{\partial \Phi_i}{\partial X_{kh}} - \dfrac{\partial \Phi_j}{\partial X_{kh}} \right) \\
%	&= \dfrac{-\left( 2 \times \1_{[k<h]} + \sqrt{2} \1_{[k=h]} \right)}{(\lambda_i - \lambda_j)^2} \left( U_{ki}U_{hi} - U_{kj}U_{hj} \right).
\end{align*}
we have, by \eqref{eq-2.4'},
\begin{align*}
	\sum_{j:j\neq i} \dfrac{\partial \Psi_{ij}}{\partial b_{kh}} (b(s,t))
	= \dfrac{1}{2} \sum_{k' \le h'} \dfrac{\partial^3 \Phi_i}{\partial b_{kh} \partial b_{k'h'}^2}.
\end{align*}
Therefore,  \eqref{eq-spde} can be written in a symmetric form: for $(S, T)\prec\prec (\sigma, \tau)$ and $1\le i\le d$, 
 \begin{align} \label{eq-spde'}
	&\lambda_i (S,T) - \lambda_i(0,0)\notag\\
	=& \sum_{k \le h} \iint_{R_{ST}} \dfrac{\partial \Phi_i}{\partial b_{kh}}(b(s,t)) db_{kh}(s,t) + \sum_{k \le h} \sum_{k' \le h'} 
	\iint_{R_{ST}} \dfrac{\partial^2 \Phi_i}{\partial b_{kh} \partial b_{k'h'}}(b(s,t)) dJ_{ b_{kh} b_{k'h'}}(s,t) \nonumber \\
	&\quad+  \frac12 \sum_{k \le h} \sum_{k' \le h'} \iint_{R_{ST}}  \dfrac{\partial^3 \Phi_i}{\partial b_{kh} \partial b_{k'h'}^2}(b(s,t)) 
	\Big( tb_{kh}(s,dt) ds+ sb_{kh}(ds,t) dt \Big)\nonumber \\
	&\quad + \sum_{j:j \neq i}\int_0^{T} \int_0^{S}  \left(\dfrac{1}{\lambda_i(s,t) - \lambda_j(s,t)} + \dfrac{2st}
	{(\lambda_i (s,t) - \lambda_j (s,t))^3} \right)ds dt.
\end{align}

Recalling that we have assumed the initial eigenvalues are distinct, by the continuity of eigenvalue functions, we have $(0,0)\prec\prec (\sigma, \tau)$ a.s. 
Thus, %\eqref{eq-spde'} can be interpreted as: 
for almost all $\omega\in \Omega$, we have the following formal partial differential equations 
near the initial point $(0,0)$: 
for $1\le i\le d$,
 \begin{align} \label{eq-spde''}
	d\lambda_i (s,t) \notag
	=& \sum_{k \le h} \dfrac{\partial \Phi_i}{\partial b_{kh}}(b(s,t)) db_{kh}(s,t) + \sum_{k \le h} \sum_{k' \le h'}\dfrac{\partial^2 \Phi_i}
	{\partial b_{kh} \partial b_{k'h'}}(b(s,t)) dJ_{ b_{kh} b_{k'h'}}(s,t) \nonumber \\
	&\quad+  \frac12 \sum_{k \le h} \sum_{k' \le h'} \dfrac{\partial^3 \Phi_i}{\partial b_{kh} \partial b_{k'h'}^2}(b(s,t)) 
	\Big( tb_{kh}(s,dt) ds+ sb_{kh}(ds,t) dt \Big)\nonumber \\
	&\quad + \sum_{j:j \neq i}\left( \dfrac{1}{\lambda_i(s,t) - \lambda_j(s,t)}+  \dfrac{2st}{(\lambda_i (s,t) - \lambda_j (s,t))^3}\right) ds dt.
\end{align}

%{\blue \begin{remark}
%Add some comments on L\'evy's characterization for Brownian sheet?
%\end{remark}}

\section{High-dimensional limit of the empirical spectral distributions}\label{sec:limit}

In this section, we study the high-dimensional limit of empirical spectral measure of the rescaled Brownian sheet matrices. In Section \ref{sec:tightness}, we first obtain the tightness of the empirical spectral measures (Theorem \ref{Thm-tightness}), and then show the convergence by Wigner's theorem (Theorem \ref{Thm-limit measure}). In Section \ref{sec:PDE}, we derive a  PDE  for the Stieltjes transform of the limiting measure and also a McKean-Vlasov equation for the limiting measure.

\subsection{Tightness and high-dimensional limit}
\label{sec:tightness}

For every integer $d \ge 1$, let $\X^d(s,t)$ be a $d\times d$ matrix given by \eqref{e:X}, and  $\{\lambda_i^d(s,t):1 \le i \le d\}$ 
be the set of eigenvalues of $\X^d(s,t)$. Define the empirical spectral measure of $\X^d(s,t)/\sqrt d$
\begin{align}\label{e:em}
	L_d(s,t)(dx)
	= \dfrac{1}{d} \sum_{i=1}^d \delta_{\lambda_i^d(s,t)/\sqrt{d}} (dx).
\end{align}
For a measurable function $g:\bR\to \bR$, we write
\begin{align} \label{eq-emprical}
	\langle g, L_d(s,t) \rangle
	:=\int_{\bR} g(x) L_d(s,t) (dx)
	= \dfrac{1}{d} \sum_{i=1}^d g \left( \dfrac{\lambda_i^d(s,t)}{\sqrt{d}} \right).
\end{align}

Let $\cP(\bR)$ be the set of probability measures 
on $\bR$ equipped with its weak topology and $C_0(\bR)$ be the set of continuous functions on $\bR$ that vanish at infinity. 
Throughout this subsection, let $S$ and $T$ be two fixed positive numbers, and recall the notation 
$R_{ST}=[0,S]\times[0,T].$

The following tightness criterion for probability-measure-valued stochastic processes is a straightforward generalization of
 \cite[Proposition B.3]{SYY20} (see also \cite[Section 3]{Rogers1993} where  this criterion was applied implicitly).
\begin{lemma} \label{Lemma-tightness}
Let $\{\mu_d(s,t), (s,t) \in R_{ST}\}_{d \in \bN}\subset C(R_{ST}, \cP(\bR))$ be a sequence of probability-measure-valued 
random fields. Assume the following conditions are satisfied:
	\begin{enumerate}
		\item[(A)] there exists a non-negative function $\varphi(x)$ satisfying $\lim\limits_{|x| \to +\infty} \varphi(x) = +\infty$ and
		\begin{align*}
			\sup_{d \in \bN} \bE \left[ \left| \left\langle \varphi, \mu_d(s,t) \right\rangle \right|^{\gamma} \right] < \infty,
			~ \forall (s,t) \in R_{ST},
		\end{align*}
		for some $\gamma > 0$;
		
		\item[(B)] there exists a countable dense subset $\{f_i(x), x\in\bR\}_{i \in \bN}$ of $C_0(\bR)$, such that for some positive constants $a_1>1$ and  $a_2>1$,
		\begin{align*}
			&\bE \left[ \left| \left\langle f_i, \mu_d(s_2,t_2) \right\rangle  - \left\langle f_i, \mu_d(s_1,t_1) \right\rangle \right|^{a_1} \right]
			\le C_{f_i,S,T} |(s_2,t_2) - (s_1,t_1)|^{a_2}
		\end{align*}
for all  $(s_1,t_1), (s_2,t_2) \in R_{ST},  d \in \bN$ and $i\in N$, where $C_{f_i,S,T}$ is a constant depending only on $S$, $T$ and $f_i$.
	\end{enumerate}
	Then the set  $\{\mu_d(s,t), (s,t) \in R_{ST}\}_{d \in \bN}$ of $C(R_{ST}, \cP(\bR))$-valued random elements is tight, i.e., it induces a tight family of probability measures on $C(R_{ST}, \cP(\bR))$.
\end{lemma}

The Kolmogorov continuity theorem for random fields (see e.g. \cite[Theorem 2.5.1 in Chapter 5]{K02}) implies that, on every compact interval, the 
Brownian sheet is $\beta$-H\"{o}lder continuous for $\beta\in(0, \frac12)$. The following lemma is a direct consequence of Fernique's theorem (\cite{fernique}). 

\begin{lemma} \label{Lemma-Holder of Bs} 
For any $\beta\in(0,\frac12)$, there exists a positive constant $\delta=\delta(\beta, S, T)$ depending only on $(\beta, S, T)$ such that
	\begin{align*}
		\bE \left[ \exp \left(\delta \left\| B \right\|_{\beta;R_{ST}}^2 \right) \right] < \infty,
	\end{align*}
where
	\begin{align} \label{eq-def-Holder of Bs}
		\left\|B \right\|_{\beta;R_{ST}}
		= \sup_{(s_1,t_1), (s_2,t_2) \in R_{ST}} \dfrac{\left| B(s_2,t_2) - B(s_1,t_1) \right|}{|(s_2,t_2)-(s_1,t_1)|^{\beta}}
	\end{align}
	is the $\beta$-H\"{o}lder norm of $B$ on the rectangle $R_{ST}$.
\end{lemma}
Now we are ready to derive the following  result on the tightness of the sequence $\{L_d(s,t), (s,t)\in R_{ST}\}_{d\in\bN}$. 
\begin{theorem} \label{Thm-tightness}
	Assume that there exists a nonnegative function $\varphi(x) \in C^1(\bR)$ with bounded  derivative, such that 
	\begin{align} \label{eq-initial condition}
	\lim\limits_{|x| \to +\infty} \varphi(x) = +\infty	~ \mbox{ and } ~ \sup_{d \in \bN} \left\langle \varphi, L_d(0,0) \right\rangle < \infty.
	\end{align}
Then the sequence $\{L_d(s,t), (s,t) \in R_{ST}\}_{d \in \bN}$ is tight on $C(R_{ST}, \cP(\bR))$.
\end{theorem}

\begin{proof}
Let $f$ be an arbitrary continuously differentiable function with bounded derivative.  By the mean value theorem and 
the Hoffman-Wielandt inequality (see e.g. \cite[Lemma 2.1.19]{anderson2010}), we have for $(s_2,t_2), (s_1,t_1) \in R_{ST}$,
	\begin{align} \label{eq-pathwise holder}
		&\left| \left\langle f, L_d(s_2,t_2) \right\rangle - \left\langle f, L_d(s_1,t_1) \right\rangle \right|^2
		=\left| \dfrac{1}{d} \sum_{i=1}^d \left( f \left( \dfrac{\lambda_i^d(s_2,t_2)}{\sqrt{d}} \right) - f \left( \dfrac{\lambda_i^d(s_1,t_1)}{\sqrt{d}} \right) \right) \right|^2 \nonumber \\
		&\le \dfrac{1}{d} \sum_{i=1}^d \left|  f \left( \dfrac{\lambda_i^d(s_2,t_2)}{\sqrt{d}} \right) - f \left( \dfrac{\lambda_i^d(s_1,t_1)}{\sqrt{d}} \right) \right|^2 
		\le \dfrac{\|f'\|_{L^{\infty}}^2}{d^2} \sum_{i=1}^d \left| \lambda_i^d(s_2,t_2) - \lambda_i^d(s_1,t_1) \right|^2 \nonumber \\
		&\le \dfrac{\|f'\|_{L^{\infty}}^2}{d^2} \sum_{i,j=1}^d \left| X_{ij}^d(s_2,t_2) - X_{ij}^d(s_1,t_1) \right|^2 
		= \dfrac{2\|f'\|_{L^{\infty}}^2}{d^2} \sum_{i \le j}\left| b_{ij}(s_2,t_2) - b_{ij}(s_1,t_1) \right|^2.
	\end{align}
	Noting that $\{b_{ij}(s,t)\}_{1 \le i \le j \le d}$ are standard Brownian sheets, by \eqref{eq-pathwise holder} and the Minkowski inequality,  
	we have for some $\beta\in (0,\frac12)$,
	\begin{align} \label{eq-holder in expectation}
		& \bE \left[ \left| \left\langle f, L_d(s_2,t_2) \right\rangle - \left\langle f, L_d(s_1,t_1) \right\rangle \right|^4 \right] \nonumber \\
		\le& \dfrac{4\|f'\|_{L^{\infty}}^4}{d^4} \bE \left[ \left( \sum_{i\le j} \left| b_{ij}(s_2,t_2) - b_{ij}(s_1,t_1) \right|^2 \right)^2 \right] \nonumber \\
		\le& \dfrac{4\|f'\|_{L^{\infty}}^4}{d^4} \left( \sum_{i\le j} \left( \bE \left[ \left| b_{ij}(s_2,t_2) - b_{ij}(s_1,t_1) \right|^4 \right] \right)^{1/2} \right)^2 \nonumber \\
		=& \dfrac{4\|f'\|_{L^{\infty}}^4}{d^4} \left( \dfrac{d(d+1)}{2} \left( \bE \left[ \left| b_{11}(s_2,t_2) - b_{11}(s_1,t_1) \right|^4 \right] \right)^{1/2} \right)^2 \nonumber \\
		=& \dfrac{(d+1)^2 \|f'\|_{L^{\infty}}^4}{d^2} \bE \left[ \left| b_{11}(s_2,t_2) - b_{11}(s_1,t_1) \right|^4 \right] \nonumber \\
		\le& 4 \|f'\|_{L^{\infty}}^4\bE \left[ \left\|b_{11} \right\|_{\beta;R_{ST}}^4 |(s_2,t_2)-(s_1,t_1)|^{4\beta} \right] \nonumber \\
%		\le& \dfrac{2 (d+1)^2 \|f'\|_{L^{\infty}}^4}{\delta^2 d^2} \bE \left[ \exp \left( \delta \left\|b_{11} \right\|_{\beta;R_{ST}}^2 \right) \right] |(s_2,t_2)-(s_1,t_1)|^{4\beta}\notag\\
		=& C(\beta, f', S,T) |(s_2,t_2)-(s_1,t_1)|^{4\beta},
	\end{align}
	where  $C(\beta, f', S,T)$ is a finite positive constant by Lemma \ref{Lemma-Holder of Bs}.

As a consequence, Condition (A) in Lemma \ref{Lemma-tightness} is satisfied with $\gamma=4$.  Moreover, if we choose $\beta \in (\frac14,\frac12)$, then assumption \eqref{eq-initial condition} and  \eqref{eq-holder in expectation} together yield Condition (B) in Lemma \ref{Lemma-tightness} with $a_1=4$, $a_2=4\beta$ and $\{f_i\}_{i \in \bN}$ being a sequence of functions in $C^1(\bR)$ with bounded derivative that is dense in $C_0(\bR)$. Then the proof is concluded by Lemma \ref{Lemma-tightness}. 	
\end{proof}
\begin{remark}
In the above proof, the independence of the Brownian sheets $b_{ij}$ ($i \le j$) actually is not used. 
\end{remark}

Denote by $\mu_{sc}(dx)$  the semicircle distribution, i.e. $\mu_{sc}(dx) = p_{sc}(x) dx$, where the density function is given by
\begin{align*}
	p_{sc}(x) = \dfrac{\sqrt{4-x^2}}{2\pi} 1_{[-2,2]}(x).
\end{align*}
Let $\{\tilde \mu(s,t), (s,t) \in R_{ST}\}$ be an element in $C(R_{ST}, \cP(\bR))$ such that $\tilde \mu(s,t)$ is a probability measure with density function $\tilde p_{s,t}(x) = \frac{1}{\sqrt{st}} p_{sc}(x/\sqrt{st})$. That is, $\tilde \mu(s,t)$ is a rescaled semicircle distribution. Here, we use the convention that $\tilde \mu(s,t)(dx) = \delta_0(dx)$ if $st=0$.

\begin{theorem} \label{Thm-limit measure}
	Assume the same condition as in Theorem \ref{Thm-tightness}. Also assume that $\{\X^d(0,0), d\in \bN\}$ are  symmetric deterministic matrices such that 
	\begin{align*}
		D := \sup_{d \in \bN} \left\| \frac{1}{\sqrt{d}} \X^d(0,0) \right\|
		< \infty,
	\end{align*}
	where $\|\cdot\|$ is the operator norm (the operator norm of a symmetric matrix is its largest eigenvalue),  and that $L_d(0,0)$ converges weakly to some probability measure $\mu(0,0)$ as $d$ goes to infinity.
	
	Then, as $d\to \infty$, $\{L_d(s,t), (s,t) \in R_{ST}\}$ converges in probability to $\{\mu(s,t), (s,t) \in R_{ST}\}$ in $C(R_{ST}, \cP(\bR))$ which is given by
	\begin{align}\label{e:limit-measure}
		\mu(s,t) = \tilde{\mu}(s,t) \boxplus \mu(0,0),
	\end{align}
where $\boxplus$ is the free additive convolution of two probability measures (\cite[Definition 5.3.20]{anderson2010}).
\end{theorem}

\begin{proof}
	For any fixed $(s,t) \in R_{ST}$ with $st > 0$, we have
	\begin{align*}
		\dfrac{1}{\sqrt{d}} \X^d(s,t)
		= \dfrac{1}{\sqrt{d}} \left( \X^d(s,t) - \X^d(0,0) \right) + \dfrac{1}{\sqrt{d}} \X^d(0,0).
	\end{align*}
By the self-similarity property of the Brownian sheet, one can see that  $\frac{1}{\sqrt{std}} \left( \X^d(s,t) - \X^d(0,0) \right)$ is a $d \times d$  
Wigner matrix (see e.g. \cite[Section 2.1]{anderson2010} for the definition). By Wigner's semicircle law (see e.g. \cite[Theorem 2.1.1]{anderson2010}), 
the  empirical spectral measure of  $\frac{1}{\sqrt{std}} \left( \X^d(s,t) - \X^d(0,0) \right)$ converges in probability to the semicircle distribution 
$\mu_{sc}$  in $\mathcal P(\bR)$ as $d \to \infty$. Thus, the empirical spectral measure of $\frac{1}{\sqrt{d}} \left( \X^d(s,t) - \X^d(0,0) \right)$ 
converges in probability to the measure $\tilde \mu(s,t)$  in $\mathcal P(\bR)$ as $d \to \infty$. Note that the  empirical spectral measure of the 
matrix $\frac{1}{\sqrt{d}} \X^d(0,0)$ is $L_d(0,0)$, which converges to $\mu(0,0)$  in $\mathcal P(\bR)$. Therefore, by \cite[Theorem 5.4.5]{anderson2010}, 
for every $(s,t) \in R_{ST}$ with $st > 0$, the empirical spectral measure of the matrix $\frac{1}{\sqrt{d}} \X^d(s,t)$ converges in probability 
to the measure $\mu(s,t)$ given by \eqref{e:limit-measure}  in $\mathcal P(\bR)$ as $d$ goes to infinity. Moreover, when $st=0$, $\frac{1}{\sqrt{d}} \X^d(s,t) = \frac{1}{\sqrt{d}} \X^d(0,0)$, and the empirical spectral measures converge to $\mu(0,0)$ in $\mathcal P(\bR)$.

	By Theorem \ref{Thm-tightness}, the sequence $\{L_d(s,t), (s,t) \in R_{ST}\}_{d \in \bN}$ is tight. Let $\{\nu(s,t), (s,t) \in R_{ST}\}$ 
	be the weak limit of an arbitrary convergent subsequence of  $\{L_d(s,t), (s,t) \in R_{ST}\}_{d \in \bN}$. Noting that  for every fixed 
	$(s,t) \in R_{ST}$, $L_d(s,t)$ is the  empirical spectral measure of the matrix $\frac{1}{\sqrt{d}} \X^d(s,t)$ and it converges 
	in probability to the deterministic measure $\mu(s,t)$, we can conclude that $\nu(s,t)=\mu(s,t)$ for $(s,t) \in R_{ST}$.
	
	Therefore, the limit of any convergent subsequence of $\{L_d(s,t), (s,t) \in R_{ST}\}_{d \in \bN}$ is the deterministic measure 
	$\{\mu(s,t), (s,t) \in R_{ST}\}$ given by \eqref{e:limit-measure}.  The proof is concluded.
\end{proof}

\subsection{PDEs for the limit measure} \label{sec:PDE}

It is known (see e.g. \cite{anderson2010}) that the high-dimensional limit $\hat \mu_t(dx)$ of the empirical measures of Dyson Brownian 
motion  \eqref{Eq:Dyson} satisfies the following McKean-Vlasov equation, 
\begin{align} \label{eq-Dyson limit equation}
	\dfrac{\partial}{\partial t} \int_{\bR} f(x) \hat \mu_t(dx)
	= \dfrac{1}{2} \iint_{\bR^2} \dfrac{f'(x) - f'(y)}{x-y} \hat \mu_t(dx) \hat \mu_t(dy), ~~\mbox{for } f \in C_b^2(\bR).
\end{align}
The Stieltjes transform 
\[\hat G_t(z)=\int_{\bR} \frac1{z-x} \hat \mu_t(dx),  \mbox{ for } z\in \mathbb C\backslash \bR\] 
of $\hat \mu_t(dx)$ solves the following complex version of inviscid Burgers' equation
\begin{align*}
	\partial_t \hat G_t(z) + \hat G_t(z) \partial_z \hat G_t(z) = 0.
\end{align*}

In this subsection, we will derive parallel PDEs for the limit $\mu(s,t)$  (see Theorem \ref{Thm-limit measure}) of the empirical spectral  
measures of  the rescaled Brownian sheet matrices.  We remark that the equations are obtained by the properties the semicircle 
distribution and may have other equivalent forms. 

Assume $\mu(0,0)(dx) = \delta_0(dx)$, then the limiting measure $\mu(s,t) (dx) = \tilde{\mu}(s,t) (dx)$, recalling that 
\[\tilde{\mu}(s,t) (dx) = \tilde p_{s,t}(x)dx=\frac{1}{\sqrt{st}} p_{sc}(x/\sqrt{st})dx\] is a rescaled semicircle distribution. Thus,  
for a test function $f$, we have
\begin{align} \label{eq-PDE}
	\dfrac{\partial^2}{\partial s \partial t} \left\langle f, \mu(s,t) \right\rangle
	=& \dfrac{\partial^2}{\partial s \partial t} \int_{\bR} \dfrac{1}{\sqrt{st}} f(x) p_{sc}(x/\sqrt{st}) dx \nonumber \\
	=& \dfrac{\partial^2}{\partial s \partial t} \int_{\bR} f(\sqrt{st} x) p_{sc}(x) dx \nonumber \\
	=& \dfrac{\partial}{\partial s} \int_{\bR} \dfrac{\sqrt{s} x}{2 \sqrt{t}} f'(\sqrt{st} x) p_{sc}(x) dx \nonumber \\
	=& \int_{\bR} \dfrac{x^2}{4} f''(\sqrt{st} x) p_{sc}(x) dx + \int_{\bR} \dfrac{x}{4 \sqrt{st}} f'(\sqrt{st} x) p_{sc}(x) dx \nonumber \\
	=& \int_{\bR} \dfrac{x^2}{4 (st)^{3/2}} f''(x) p_{sc}(x/\sqrt{st}) dx + \int_{\bR} \dfrac{x}{4 (st)^{3/2}} f'(x) p_{sc}(x/\sqrt{st}) dx \nonumber \\
	=& \dfrac{1}{4st} \left\langle x^2 f''(x), \mu(s,t) \right\rangle + \dfrac{1}{4st} \left\langle x f'(x), \mu(s,t) \right\rangle.
\end{align}

Noting that the  density of the measure $\hat \mu_t(dx)$ is $\hat p_t(x) = \frac{1}{\sqrt{t}} p_{sc}(x/\sqrt{t})$, the left-hand side of 
\eqref{eq-Dyson limit equation} can be written as
\begin{align} \label{eq-limit left}
	\dfrac{\partial}{\partial t} \int_{\bR} f(x) \hat \mu_t(dx)
	=& \dfrac{\partial}{\partial t} \int_{\bR} \dfrac{1}{\sqrt{t}} f(x) p_{sc}(x/\sqrt{t}) dx \nonumber \\
	=& \dfrac{\partial}{\partial t} \int_{\bR} f(\sqrt{t} x) p_{sc}(x) dx \nonumber \\
	=& \int_{\bR} \dfrac{x}{2\sqrt{t}} f'(\sqrt{t} x) p_{sc}(x) dx.
\end{align}
Similarly, the right-hand side of \eqref{eq-Dyson limit equation} can be written as
\begin{align} \label{eq-limit right}
	\dfrac{1}{2} \iint_{\bR^2} \dfrac{f'(x) - f'(y)}{x-y} \hat \mu_t(dx) \hat \mu_t(dy)
	= \dfrac{1}{2} \iint_{\bR^2} \dfrac{f'(\sqrt{t} x) - f'(\sqrt{t} y)}{\sqrt{t} (x-y)} p_{sc}(x) p_{sc}(y) dxdy.
\end{align}
Substituting \eqref{eq-limit left} and \eqref{eq-limit right} into \eqref{eq-Dyson limit equation}, we get
\begin{align} \label{eq-semicircle identity}
	\int_{\bR} x f'(\sqrt{t} x) p_{sc}(x) dx
	= \iint_{\bR^2} \dfrac{f'(\sqrt{t} x) - f'(\sqrt{t} y)}{x-y} p_{sc}(x) p_{sc}(y) dxdy, ~ \forall t > 0.
\end{align}
Taking derivative with respect to $t$ for both sides, we have
\begin{align} \label{eq-semicircle identity derivative}
	\int_{\bR} x^2 f''(\sqrt{t} x) p_{sc}(x) dx
	= \iint_{\bR^2} \dfrac{xf''(\sqrt{t} x) - yf''(\sqrt{t} y)}{x-y} p_{sc}(x) p_{sc}(y) dxdy, ~ \forall t > 0.
\end{align}
Now, combining \eqref{eq-PDE}, \eqref{eq-semicircle identity} and \eqref{eq-semicircle identity derivative}, we have
\begin{align}
	\dfrac{\partial^2}{\partial s \partial t} \left\langle f, \mu(s,t) \right\rangle
	=& \int_{\bR} \dfrac{x^2}{4} f''(\sqrt{st} x) p_{sc}(x) dx
	+ \int_{\bR} \dfrac{x}{4 \sqrt{st}} f'(\sqrt{st} x) p_{sc}(x) dx \nonumber \\
	=& \dfrac{1}{4} \iint_{\bR^2} \dfrac{xf''(\sqrt{st} x) - yf''(\sqrt{st} y)}{x-y} p_{sc}(x) p_{sc}(y) dxdy \nonumber \\
	&+ \dfrac{1}{4\sqrt{st}} \iint_{\bR^2} \dfrac{f'(\sqrt{st} x) - f'(\sqrt{st} y)}{x-y} p_{sc}(x) p_{sc}(y) dxdy \nonumber \\
	=& \dfrac{1}{4} \iint_{\bR^2} \dfrac{xf''(x) - yf''(y)}{x-y} \cdot \dfrac{1}{st} p_{sc}(x/\sqrt{st}) p_{sc}(y/\sqrt{st}) dxdy \nonumber \\
	&+ \dfrac{1}{4} \iint_{\bR^2} \dfrac{f'(x) - f'(y)}{x-y} \cdot \dfrac{1}{st} p_{sc}(x/\sqrt{st}) p_{sc}(y/\sqrt{st}) dxdy \nonumber \\
	=& \dfrac{1}{4} \iint_{\bR^2} \dfrac{xf''(x) - yf''(y)}{x-y} \mu(s,t)(dx) \mu(s,t)(dy) \nonumber \\
	&+ \dfrac{1}{4} \iint_{\bR^2} \dfrac{f'(x) - f'(y)}{x-y} \mu(s,t)(dx) \mu(s,t)(dy).
\end{align}
Therefore, we get the following McKean-Vlasov equation for $\mu(s,t)(dx)$:
\begin{equation}\label{e:MV}
\dfrac{\partial^2}{\partial s \partial t} \left\langle f, \mu(s,t) \right\rangle=
\frac14 \iint_{\bR^2} \frac{\left(xf'(x)\right)'-\left(yf'(y)\right)'}{x-y} (\mu(s,t))^{\otimes 2}(dx,dy).
\end{equation}

Now we consider the Stieltjes transform of $\mu(s,t)(dx)$:
\begin{align*}
	G_{s,t}(z) = \left\langle \dfrac{1}{z-x}, \mu(s,t) \right\rangle, \mbox{ for } z \in \bC\backslash \bR.
\end{align*}
Note that the Stieltjes transform $G(z)$ of the semicircle distribution $p_{sc}(x) dx$ can be written as
\begin{align} \label{eq-08}
	G(z) = \left\langle (z-x)^{-1}, \mu_{sc} \right\rangle
	=& \int_{\bR} \dfrac{1}{z-x} p_{sc}(x) dx \nonumber \\
	=& \int_{\bR} \dfrac{1}{z-x/\sqrt{st}} p_{sc}(x/\sqrt{st}) \dfrac{dx}{\sqrt{st}} \nonumber \\
	=& \int_{\bR} \dfrac{\sqrt{st}}{\sqrt{st}z-x} \tilde p_{s,t}(x) dx \nonumber \\
	=& \sqrt{st} G_{s,t}(\sqrt{st} z).
\end{align}
By \cite[(2.4.6)]{anderson2010} (see also \cite[(2.103)]{Tao2012}), $G(z)$ solves 
\begin{align} \label{eq-09}
	G(z)^2 - z G(z) + 1 = 0.
\end{align}
Substituting \eqref{eq-08} into \eqref{eq-09}, we have
\begin{align*}
	st \left( G_{s,t}(\sqrt{st} z) \right)^2 - z \sqrt{st} G_{s,t}(\sqrt{st} z) + 1 = 0,
\end{align*}
which can be rewritten as
\begin{align} \label{eq-derivative 0 order}
	st \left( G_{s,t}(z) \right)^2 - z G_{s,t}(z) + 1 = 0.
\end{align}
Taking the derivative with respect to $z$ in \eqref{eq-derivative 0 order}, we get 
\begin{align} \label{eq-derivative 1 order}
	2st G_{s,t}(z) \partial_z G_{s,t}(z) - z \partial_z G_{s,t}(z) - G_{s,t}(z) = 0.
\end{align}
Take the derivative with respect to $z$ in \eqref{eq-derivative 1 order}, we have
\begin{align} \label{eq-derivative 2 order}
	2st \left( G_{s,t}(z) \partial_z^2 G_{s,t}(z) + \left( \partial_z G_{s,t}(z) \right)^2 \right) - z \partial_z^2 G_{s,t}(z) - 2\partial_z G_{s,t}(z) = 0.
\end{align}
Now, by choosing $f(x) = (z-x)^{-1}$ in \eqref{eq-PDE}, we have
\begin{align} \label{eq-Stieltjes transform-PDE}
	\dfrac{\partial^2}{\partial s \partial t} G_{s,t}(z)
	=& \dfrac{1}{2st} \left\langle \dfrac{x^2}{(z-x)^3}, \mu(s,t) \right\rangle
	+ \dfrac{1}{4st} \left\langle \dfrac{x}{(z-x)^2}, \mu(s,t) \right\rangle \nonumber \\
	=& \dfrac{1}{2st} \left\langle \dfrac{(z-x)^2 - 2z(z-x) + z^2}{(z-x)^3}, \mu(s,t) \right\rangle
	+ \dfrac{1}{4st} \left\langle \dfrac{(x-z) + z}{(z-x)^2}, \mu(s,t) \right\rangle \nonumber \\
	=& \dfrac{1}{4st} \left\langle \dfrac{1}{z-x}, \mu(s,t) \right\rangle
	- \dfrac{3z}{4st} \left\langle \dfrac{1}{(z-x)^2}, \mu(s,t) \right\rangle
	+ \dfrac{z^2}{2st} \left\langle \dfrac{1}{(z-x)^3}, \mu(s,t) \right\rangle \nonumber \\
	=& \dfrac{1}{4st} G_{s,t}(z) + \dfrac{3z}{4st} \partial_z G_{s,t}(z) + \dfrac{z^2}{4st} \partial_z^2 G_{s,t}(z) \nonumber \\
	=& \dfrac{1}{4st} \big( G_{s,t}(z) + z \partial_z G_{s,t}(z) \big)
	+ \dfrac{z}{4st} \big( 2 \partial_z G_{s,t}(z) + z \partial_z^2 G_{s,t}(z) \big) \nonumber \\
	=& \dfrac{1}{2} G_{s,t}(z) \partial_z G_{s,t}(z) + \dfrac{z}{2} \left( G_{s,t}(z) \partial_z^2 G_{s,t}(z) + \left( \partial_z G_{s,t}(z) \right)^2 \right),
\end{align}
where the last equality follows from \eqref{eq-derivative 1 order} and \eqref{eq-derivative 2 order}. Therefore, we have the following 
generalized Burgers' equation for $G_{s,t}(z)$ 
\begin{equation}\label{e:Burgers'}
	\dfrac{\partial^2}{\partial s \partial t} G_{s,t}(z)=\dfrac{1}{2} G_{s,t}(z) \partial_z G_{s,t}(z) + \dfrac{z}{2} 
	\left( G_{s,t}(z) \partial_z^2 G_{s,t}(z) + \left( \partial_z G_{s,t}(z) \right)^2 \right).
\end{equation}

\appendix
\section{Some lemmas in matrix calculus}\label{sec:appendix}

In this Appendix, we provide some results in matrix analysis which are used in Sections 3 and 4.

\begin{lemma}\label{lem1}
Let $a=(a_1,\dots, a_d), b=(b_1,\dots,b_d)$ be two $d$-dimensional vectors such that $\|a\|=\|b\|=1$ and $a\cdot b=0$. Then 
\[\sum_{1\le i,j\le d}(a_i b_j+a_jb_i)^2=\sum_{1\le i,j\le d}(a_i a_j+b_ib_j)^2=2.\]
\end{lemma}
\begin{proof} This is elementary to verify:
\begin{align*}
\sum_{1\le i,j\le d}(a_i b_j+a_jb_i)^2=\sum_{i,j} \left(a_i^2b_j^2+a_j^2 b_i^2 +2 a_ib_i a_jb_j\right)=2\|a\|^2\|b\|^2+2(a\cdot b)^2=2.
\end{align*}
Similarly, one can show $\sum\limits_{1\le i,j\le d}(a_i a_j+b_ib_j)^2=2$.
\end{proof}

For a $d \times d$ real symmetric matrix $X = \left( X_{ij} \right)$, we write $X = UDU^\T$, where $U$ is an orthogonal matrix 
and $D = \diag(\lambda_1, \ldots, \lambda_d)$. Noting that the space of $d\times d$ symmetric matrices can be identified with 
$\bR^{d(d+1)/2}$, we consider the $i$-th biggest eigenvalue $\lambda_i=\tilde \Phi_i(X)$ as a function of $d(d+1)/2$ variables 
$(X_{kh},1\le  k\le h\le d)$ for $i=1,\dots, d$.

\begin{lemma} \label{Lemma-eigen derivative}
Suppose that $X$ is a smooth function of parameters $\theta, \xi \in \bR$. Then we have
\begin{align} \label{eq-lemma 1-1}
	&\partial_{\theta} \lambda_i =\left( U^{\T} \partial_{\theta}X U \right)_{ii},
\end{align}
\begin{align} \label{eq-lemma 1-2}
	\partial_{\xi} \partial_{\theta} \lambda_i
	= \left( U^{\T} \partial_{\xi} \partial_{\theta}X U \right)_{ii} + 2 \sum_{j:j\neq i} \dfrac{\left( U^{\T} \partial_{\theta}X U \right)_{ij} 
	\left( U^{\T} \partial_{\xi}X U \right)_{ij} }{\lambda_i-\lambda_j}
\end{align}
and
\begin{align} \label{eq-lemma 1-3}
&\partial_{\xi}^2 \partial_{\theta} \lambda_i
	=\left( U^{\T} \partial_{\xi}^2 \partial_{\theta}X U \right)_{ii}
		+ \sum_{j:j\neq i} \dfrac{4\left( U^{\T} \partial_{\xi} \partial_{\theta}X U \right)_{ij} \left( U^{\T} \partial_{\xi}X U \right)_{ij} + 2\left( U^{\T} \partial_{\theta}X U \right)_{ij} \big( U^{\T} \partial_{\xi}^2X U \big)_{ij}}{\lambda_i-\lambda_j} \nonumber \\
	&\qquad + 2 \sum_{j:j\neq i} \Bigg[\sum_{l:l\neq i} \dfrac{\left( U^{\T} \partial_{\xi}X U \right)_{il} \left( U^{\T} \partial_{\theta}X U \right)_{lj} \left( U^{\T} \partial_{\xi}X U \right)_{ij}}{\left( \lambda_i-\lambda_l \right) \left( \lambda_i-\lambda_j \right)}
	+ \sum_{l:l\neq i} \dfrac{ \left( U^{\T} \partial_{\xi}X U \right)_{il} \left( U^{\T} \partial_{\xi}X U \right)_{lj} \left( U^{\T} \partial_{\theta}X U \right)_{ij}}{\left( \lambda_i-\lambda_l \right) \left( \lambda_i-\lambda_j \right)} \nonumber \\
	&\qquad + \sum_{l:l\neq j} \dfrac{ \left( U^{\T} \partial_{\theta}X U \right)_{il} \left( U^{\T} \partial_{\xi}X U \right)_{lj}  \left( U^{\T} \partial_{\xi}X U \right)_{ij}}{\left( \lambda_j-\lambda_l \right) \left( \lambda_i-\lambda_j \right)}  +  \sum_{l:l\neq j} \dfrac{ \left( U^{\T} \partial_{\xi}X U \right)_{il} \left( U^{\T} \partial_{\xi}X U \right)_{lj} \left( U^{\T} \partial_{\theta}X U \right)_{ij}}{\left( \lambda_j-\lambda_l \right) \left( \lambda_i-\lambda_j \right)}
	\nonumber \\
	&\qquad  - \dfrac{\left( U^{\T} \partial_{\theta}X U \right)_{ij} \left( U^{\T} \partial_{\xi}X U \right)_{ij}}{\left( \lambda_i-\lambda_j \right)^2} 
	\left( \left( U^{\T} \partial_{\xi}X U \right)_{ii} - \left( U^{\T} \partial_{\xi}X U \right)_{jj} \right)\Bigg].
\end{align}
\end{lemma}

\begin{proof}
Since $D = U^{\T}XU$, we have
\begin{align} \label{eq-1st derivative}
	\partial_{\theta} D
	= \partial_{\theta}U^{\T} XU + U^{\T} \partial_{\theta}X U + U^{\T} X \partial_{\theta}U
	= \partial_{\theta}U^{\T} UD + U^{\T} \partial_{\theta}X U + D U^{\T} \partial_{\theta}U.
\end{align}
Besides,
\begin{align} \label{eq-derivative of U}
	0_d = \partial_{\theta} I_d
	= \partial_{\theta} \left( U^{\T}U \right)
	= \partial_{\theta}U^{\T} U + U^{\T} \partial_{\theta}U.
\end{align}
In particular, this implies 
\begin{equation}\label{e:A6}
\left(\partial_\theta U^\T U\right)_{ii}=\left(U^\T\partial_\theta U\right)_{ii}=0, ~1\le i\le d. 
\end{equation}
The first identity \eqref{eq-lemma 1-1} follows from the diagonal entries of \eqref{eq-1st derivative} and \eqref{eq-derivative of U}.

Now we deduce \eqref{eq-lemma 1-2}. By \eqref{eq-1st derivative},
\begin{align} \label{eq-2nd derivative}
	\partial_{\xi} \partial_{\theta} D
	&= \partial_{\xi} \partial_{\theta}U^{\T} U D + \partial_{\theta}U^{\T} \partial_{\xi}U D + \partial_{\theta}U^{\T} U \partial_{\xi}D \nonumber \\
	&\quad + \partial_{\xi}U^{\T} \partial_{\theta}X U + U^{\T} \partial_{\xi} \partial_{\theta}X U + U^{\T} \partial_{\theta}X \partial_{\xi}U \nonumber \\
	&\quad + \partial_{\xi} D U^{\T} \partial_{\theta}U + D \partial_{\xi}U^{\T} \partial_{\theta}U + D U^{\T} \partial_{\xi} \partial_{\theta}U.
\end{align}
By \eqref{eq-derivative of U}, we have
\begin{align} \label{eq-1.6}
	\left( \partial_{\theta}U^{\T} U \partial_{\xi}D + \partial_{\xi} D U^{\T} \partial_{\theta}U \right)_{ii}
	= \partial_{\xi}\lambda_i \left( \partial_{\theta}U^{\T} U + U^{\T} \partial_{\theta}U \right)_{ii}
	= 0.
\end{align}
Furthermore, taking partial derivative $\partial_\xi $ on both sides of \eqref{eq-derivative of U} yields
\begin{align} \label{eeq-2 derivative of U}
	0_d = \partial_{\xi} \left( \partial_{\theta}U^{\T} U + U^{\T} \partial_{\theta}U \right)
	= \partial_{\xi} \partial_{\theta}U^{\T} U + \partial_{\theta}U^{\T} \partial_{\xi}U + \partial_{\xi}U^{\T} \partial_{\theta}U + U^{\T} \partial_{\xi} \partial_{\theta}U,
\end{align}
which implies
\begin{align} \label{eq-1.5}
	& \left( \partial_{\xi} \partial_{\theta}U^{\T} U D + \partial_{\theta}U^{\T} \partial_{\xi}U D + D \partial_{\xi}U^{\T} \partial_{\theta}U + D U^{\T} \partial_{\xi} \partial_{\theta}U \right)_{ii} \nonumber \\
	&= \lambda_i \left( \partial_{\xi} \partial_{\theta}U^{\T} U + \partial_{\theta}U^{\T} \partial_{\xi}U + \partial_{\xi}U^{\T} \partial_{\theta}U + U^{\T} \partial_{\xi} \partial_{\theta}U \right)_{ii}
	= 0.
\end{align}
Combining \eqref{eq-2nd derivative}, \eqref{eq-1.5} and \eqref{eq-1.6}, we have
\begin{align} \label{eq-1.7+}
	\partial_{\xi} \partial_{\theta} \lambda_i
	= \left( \partial_{\xi}U^{\T} \partial_{\theta}X U + U^{\T} \partial_{\xi} \partial_{\theta}X U + U^{\T} \partial_{\theta}X \partial_{\xi}U \right)_{ii}.
\end{align}

Note that the matrix identity \eqref{eq-1st derivative} is also valid when $\theta$ is replaced by $\xi$. Therefore, the non-diagonal term is
\begin{align} \label{eq-1.6+}
	0 &= \lambda_j \left( \partial_{\xi}U^{\T} U \right)_{ij} + \left( U^{\T} \partial_{\xi}X U \right)_{ij} + \lambda_i \left( U^{\T} \partial_{\xi}U \right)_{ij} \nonumber \\
	&= \left( U^{\T} \partial_{\xi}X U \right)_{ij} + \left( \lambda_i - \lambda_j \right) \left( U^{\T} \partial_{\xi}U \right)_{ij}, \forall 1 \le i \neq j \le d,
\end{align}
where the second equality follows from \eqref{eq-derivative of U}. Thus, by \eqref{eq-1.6+} and \eqref{eq-derivative of U},
\begin{align} \label{eq-1.7}
	&\left( \partial_{\xi}U^{\T} \partial_{\theta}X U + U^{\T} \partial_{\theta}X \partial_{\xi}U \right)_{ii}  \nonumber \\
	&= \left( \partial_{\xi}U^{\T} U U^{\T} \partial_{\theta}X U + U^{\T} \partial_{\theta}X U U^{\T} \partial_{\xi}U \right)_{ii} \nonumber \\
	&= \sum_{j=1}^d \left( \left( \partial_{\xi}U^{\T} U \right)_{ij} \left( U^{\T} \partial_{\theta}X U \right)_{ji} + \left( U^{\T} \partial_{\theta}X U \right)_{ij} \left( U^{\T} \partial_{\xi}U \right)_{ji} \right) \nonumber \\
	&= \sum_{j:j\neq i} \dfrac{ \left( U^{\T} \partial_{\theta}X U \right)_{ji} \left( U^{\T} \partial_{\xi}X U \right)_{ij} +\left( U^{\T} \partial_{\theta}X U \right)_{ij} \left( U^{\T} \partial_{\xi}X U \right)_{ji} }{\lambda_i-\lambda_j}.
\end{align}
Substituting \eqref{eq-1.7} into \eqref{eq-1.7+} and noting the symmetry of the matrices $U^\T \partial_\theta X U$ and $U^\T \partial_\xi X U$,  
we obtain the second identity \eqref{eq-lemma 1-2}.

Finally, we deal with \eqref{eq-lemma 1-3}.  Taking $\partial_\xi$ for  the first term on the right-hand side of \eqref{eq-lemma 1-2},  we have   by \eqref{eq-1.6+} and \eqref{eq-derivative of U},
\begin{align} \label{eq-1.15}
	& \partial_{\xi} \left( U^{\T} \partial_{\xi} \partial_{\theta}X U \right)_{ii} \nonumber \\
	&= \left( \partial_{\xi}U^{\T} \partial_{\xi} \partial_{\theta}X U \right)_{ii}
	+ \left( U^{\T} \partial_{\xi}^2 \partial_{\theta}X U \right)_{ii}
	+ \left( U^{\T} \partial_{\xi} \partial_{\theta}X \partial_{\xi}U \right)_{ii} \nonumber \\
	&= \left( U^{\T} \partial_{\xi}^2 \partial_{\theta}X U \right)_{ii}
	+ \sum_{j=1}^d \left( \partial_{\xi}U^{\T} U \right)_{ij} \left( U^{\T} \partial_{\xi} \partial_{\theta}X U \right)_{ji}
	+ \sum_{j=1}^d \left( U^{\T} \partial_{\xi} \partial_{\theta}X U \right)_{ij} \left( U^{\T} \partial_{\xi}U \right)_{ji} \nonumber \\
	&= \left( U^{\T} \partial_{\xi}^2 \partial_{\theta}X U \right)_{ii}
	+ \sum_{j:j\neq i} \dfrac{\left( U^{\T} \partial_{\xi}X U \right)_{ij} \left( U^{\T} \partial_{\xi} \partial_{\theta}X U \right)_{ji} + \left( U^{\T} \partial_{\xi} \partial_{\theta}X U \right)_{ij} \left( U^{\T} \partial_{\xi}X U \right)_{ji}}{\lambda_i-\lambda_j}\notag\\
	&= \left( U^{\T} \partial_{\xi}^2 \partial_{\theta}X U \right)_{ii}
	+2 \sum_{j:j\neq i} \dfrac{\left( U^{\T} \partial_{\xi}X U \right)_{ij} \left( U^{\T} \partial_{\xi} \partial_{\theta}X U \right)_{ij} }{\lambda_i-\lambda_j}.
\end{align}

Similarly,  it follows from \eqref{eq-1.6+},  \eqref{eq-derivative of U} and \eqref{e:A6} that
\begin{align} \label{eq-1.13}
	&\partial_{\xi} \left( U^{\T} \partial_{\theta}X U \right)_{ij} \nonumber \\
	&= \left( \partial_{\xi}U^{\T} \partial_{\theta}X U \right)_{ij} + \left( U^{\T} \partial_{\xi} \partial_{\theta}X U \right)_{ij} + \left( U^{\T} \partial_{\theta}X \partial_{\xi}U \right)_{ij} \nonumber \\
	&= \left( U^{\T} \partial_{\xi} \partial_{\theta}X U \right)_{ij} + \sum_{l=1}^d \left( \partial_{\xi}U^{\T} U \right)_{il} \left( U^{\T} \partial_{\theta}X U \right)_{lj} + \sum_{l=1}^d \left( U^{\T} \partial_{\theta}X U \right)_{il} \left( U^{\T} \partial_{\xi}U \right)_{lj} \nonumber \\
	&= \left( U^{\T} \partial_{\xi} \partial_{\theta}X U \right)_{ij}
	+ \sum_{l:l\neq i} \dfrac{\left( U^{\T} \partial_{\xi}X U \right)_{il} \left( U^{\T} \partial_{\theta}X U \right)_{lj}}{\lambda_i-\lambda_l}
	+ \sum_{l:l\neq j} \dfrac{\left( U^{\T} \partial_{\theta}X U \right)_{il}\left( U^{\T} \partial_{\xi}X U \right)_{lj} }{\lambda_j-\lambda_l}\notag\\
	&\qquad  + \left(U^{\T} \partial_\theta X U\right)_{ij} \left[\left(\partial_\xi U^{\T} U \right)_{ii}+ \left( U^{\T} \partial_\xi U \right)_{jj} \right]\notag\\&= \left( U^{\T} \partial_{\xi} \partial_{\theta}X U \right)_{ij}
	+ \sum_{l:l\neq i} \dfrac{\left( U^{\T} \partial_{\xi}X U \right)_{il} \left( U^{\T} \partial_{\theta}X U \right)_{lj}}{\lambda_i-\lambda_l}
	+ \sum_{l:l\neq j} \dfrac{ \left( U^{\T} \partial_{\theta}X U \right)_{il}\left( U^{\T} \partial_{\xi}X U \right)_{lj}}{\lambda_j-\lambda_l}. 
\end{align}

Now we deal with the second term on the right-hand side of \eqref{eq-lemma 1-2}.  By \eqref{eq-1.13} and \eqref{eq-lemma 1-1},
\begin{align} \label{eq-1.14}
	&\partial_{\xi} \bigg( \dfrac{\left( U^{\T} \partial_{\theta}X U \right)_{ij} \left( U^{\T} \partial_{\xi}X U \right)_{ij}}{\lambda_i-\lambda_j} \bigg) \nonumber \\
	&= \dfrac{\partial_{\xi} \left( U^{\T} \partial_{\theta}X U \right)_{ij} \left( U^{\T} \partial_{\xi}X U \right)_{ij} + \left( U^{\T} \partial_{\theta}X U \right)_{ij} \partial_{\xi} \left( U^{\T} \partial_{\xi}X U \right)_{ij}}{\lambda_i-\lambda_j} \nonumber \\
	&\qquad - \dfrac{\left( U^{\T} \partial_{\theta}X U \right)_{ij} \left( U^{\T} \partial_{\xi}X U \right)_{ij} }{\left( \lambda_i-\lambda_j \right)^2} \left( \partial_{\xi}\lambda_i - \partial_{\xi}\lambda_j \right) \nonumber \\
	&= \dfrac{\left( U^{\T} \partial_{\xi} \partial_{\theta}X U \right)_{ij} \left( U^{\T} \partial_{\xi}X U \right)_{ij} + \left( U^{\T} \partial_{\theta}X U \right)_{ij} \left( U^{\T} \partial_{\xi}^2X U \right)_{ij}}{\lambda_i-\lambda_j} \nonumber \\
	&\qquad + \sum_{l:l\neq i} \dfrac{\left( U^{\T} \partial_{\xi}X U \right)_{il} \left( U^{\T} \partial_{\theta}X U \right)_{lj} \left( U^{\T} \partial_{\xi}X U \right)_{ij}}{\left( \lambda_i-\lambda_l \right) \left( \lambda_i-\lambda_j \right)}
	+ \sum_{l:l\neq j} \dfrac{\left( U^{\T} \partial_{\theta}X U \right)_{il}\left( U^{\T} \partial_{\xi}X U \right)_{lj}  \left( U^{\T} \partial_{\xi}X U \right)_{ij}}{\left( \lambda_j-\lambda_l \right) \left( \lambda_i-\lambda_j \right)} \nonumber \\
	&\qquad +\sum_{l:l\neq i} \dfrac{\left( U^{\T} \partial_{\theta}X U \right)_{ij} \left( U^{\T} \partial_{\xi}X U \right)_{il} \left( U^{\T} \partial_{\xi}X U \right)_{lj}}{\left( \lambda_i-\lambda_l \right) \left( \lambda_i-\lambda_j \right)} + \sum_{l:l\neq j} \dfrac{\left( U^{\T} \partial_{\theta}X U \right)_{ij} \left( U^{\T} \partial_{\xi}X U \right)_{il} \left( U^{\T} \partial_{\xi}X U \right)_{lj}}{\left( \lambda_j-\lambda_l \right) \left( \lambda_i-\lambda_j \right)}
	\nonumber \\
	&\qquad - \dfrac{\left( U^{\T} \partial_{\theta}X U \right)_{ij} \left( U^{\T} \partial_{\xi}X U \right)_{ij} }{\left( \lambda_i-\lambda_j \right)^2} 
	\left( \left( U^{\T} \partial_{\xi}X U \right)_{ii} - \left( U^{\T} \partial_{\xi}X U \right)_{jj} \right).
\end{align}
Then the third equality \eqref{eq-lemma 1-3} follows from \eqref{eq-lemma 1-2}, \eqref{eq-1.15} and \eqref{eq-1.14}.
\end{proof}

In particular, if we choose $\theta = X_{kh}$, we have for $1\le i, j\le d$,
\begin{align}\label{e:A18}
\left(U^\T \partial_\theta X U\right)_{ij} &= \left(U_{ki}U_{hj}+U_{hi}U_{kj}\right)\1_{[k\neq h]}+ U_{ki}U_{kj}\1_{[k=h]}\notag\\
&= \left(U_{ki}U_{hj}+U_{hi}U_{kj}\right)\left(\1_{[k\neq h]} +\1_{[k=h]}/2\right). 
\end{align}
Applying \eqref{e:A18} to Lemma \ref{Lemma-eigen derivative} yields 
 \begin{align} \label{eq-1.1}
	\dfrac{\partial \lambda_i}{\partial X_{kh}}
	= 2U_{ki} U_{hi} \1_{[k \neq h]} + U_{ki}^2 \1_{[k=h]},
\end{align}
\begin{align} \label{eq-1.2}
	\dfrac{\partial^2 \lambda_i}{\partial X_{kh}^2}
	= 2 \sum_{j:j \neq i} \dfrac{\left| U_{ki} U_{hj} + U_{hi} U_{kj} \right|^2}{\lambda_i - \lambda_j} \1_{[k \neq h]} + 2 \sum_{j:j \neq i} 
	\dfrac{\left| U_{ki} U_{kj} \right|^2}{\lambda_i - \lambda_j} \1_{[k=h]},
\end{align}
and
\begin{align} \label{eq-1.14-cross derivative}
&\dfrac{\partial^2 \lambda_i}{\partial X_{kh} \partial X_{k'h'}}\notag\\
&= 2 \sum_{j:j\neq i} \dfrac{(U_{ki}U_{hj} + U_{hi}U_{kj})(\1_{[k\neq h]} + \1_{[k=h]}/2) (U_{k'i}U_{h'j} + U_{h'i}U_{k'j})(\1_{[k'\neq h']} 
+ \1_{[k'=h']}/2)}{\lambda_i - \lambda_j}. 
\end{align}

%\begin{remark}
%Note that \eqref{eq-1.1} and \eqref{eq-1.2} are also obtained in \cite{nualart2014eigenvalue} without details. {\blue Besides, the index in \cite{nualart2014eigenvalue} seems not right. We need to interchange the two indices of $U$ in \cite{nualart2014eigenvalue}.}
%\end{remark}

Recall  that   $\lambda_i=\tilde \Phi_i(X)=\tilde \Phi_i$ is the $i$-th biggest eigenvalue of $X$ and that
 $$\tilde \Psi_{ij}=\tilde \Psi_{ij}(X)
=\frac1{\lambda_i-\lambda_j}=\frac{1}{\tilde \Phi_i(X)-\tilde \Phi_j(X)}. $$
Consider a symmetric matrix $(b_{kh})_{d\times d}$. Let $x_{kh}=b_{kh}\1_{[k\neq h]}+ \sqrt{2}b_{kh}\1_{[k=h]}$ and define 
$\Phi_i=\Phi_i(b):=\tilde \Phi_i(X)$ for $i=1,\dots, d$. 
Thus by the chain rule,  we have for $1\le i, k, h\le d$, \[\frac{\partial \Phi_i}{\partial b_{kh}} = \frac{\partial \tilde \Phi_i}{\partial X_{kh}}\1_{[k\neq h]} 
+ \sqrt 2\frac{\partial \tilde \Phi_i}{\partial X_{kh}}\1_{[k=h]} = \frac{\partial \lambda_i}{\partial X_{kh}}\1_{[k\neq h]} 
+ \sqrt 2\frac{\partial \lambda_i}{\partial X_{kh}}\1_{[k=h]}.\] 
We also define 
\[\Psi_{ij}=\Psi_{ij}(b)=\tilde \Psi_{ij}(X)=\frac1{\lambda_i-\lambda_j}=\frac{1}{\Phi_i(b)-\Phi_j(b)}.
\]

The following lemma is concerned with partial derivatives of $\Phi_i(b)$ and $\Psi_{ij}(b)$.

\begin{lemma} \label{Lemma-derivative identity}
%\begin{align} \label{eq-2.2'}
%	\sum_{k \le h} \dfrac{\partial \Phi_i}{\partial b_{kh}} \dfrac{\partial \Phi_j}{\partial b_{kh}}
%	= 2 \, \1_{[i=j]},
%\end{align}
\begin{align} \label{eq-2.2}
	\sum_{k \le h} \dfrac{\partial^2 \Phi_i}{\partial b_{kh}^2}
	= 2 \sum_{j:j \neq i} \dfrac{1}{\lambda_i - \lambda_j},
\end{align}
%\begin{align} \label{eq-3.4'}
%	\sum_{k' \le h'} \dfrac{\partial^2 \Phi_i}{\partial b_{kh} \partial b_{k'h'}} \dfrac{\partial \Phi_i}{\partial b_{k'h'}}
%	= 0,
%\end{align}
%\begin{align} \label{eq-2.4cross}
%	\sum_{k \le h} \sum_{k' \le h'} \left( \dfrac{\partial^2 \Phi_i}{\partial b_{kh} \partial b_{k'h'}} \right) \left( \dfrac{\partial^2 
%	\Phi_{i'}}{\partial b_{kh} \partial b_{k'h'}} \right)
%	= \sum_{j:j\neq i} \dfrac{4}{(\lambda_i - \lambda_j)^2} \1_{[i=i']} - \dfrac{4}{(\lambda_i-\lambda_{i'})^2} \1_{[i\neq i']},
%\end{align}
\begin{align} \label{eq-2.4'}
	\sum_{k' \le h'} \dfrac{\partial^3 \Phi_i}{\partial b_{kh} \partial b_{k'h'}^2}
	= &2\left( 2 \times \1_{[k<h]} + \sqrt{2} \times \1_{[k=h]} \right)  \sum_{j:j\neq i} \dfrac{U_{kj} U_{hj} - U_{ki} U_{hi}}{(\lambda_i - \lambda_j)^2}\notag\\
	=&2 \sum_{j:j\neq i} \dfrac{1}{(\lambda_i - \lambda_j)^2} \left( \dfrac{\partial \Phi_j}{\partial b_{kh}} - \dfrac{\partial \Phi_i}{\partial b_{kh}} \right),
\end{align}

\begin{align} \label{eq-2.11}
	\sum_{k \le h} \dfrac{\partial^2 \Psi_{ij}}{\partial b_{kh}^2}
	= \dfrac{4}{(\lambda_i - \lambda_j)^3}
	+ \dfrac{1}{(\lambda_i - \lambda_j)} \sum_{l:l \neq i,j} \dfrac{2}{\left( \lambda_i - \lambda_l \right) \left( \lambda_j - \lambda_l \right)}, \text{ for } i\neq j. 
\end{align}

%\begin{align} \label{eq-A.25}
%	\sum_{k \le h} \dfrac{\partial \Phi_i}{\partial b_{kh}} \dfrac{\partial \Psi_{ij}}{\partial b_{kh}}= - \dfrac{2}{(\lambda_i-\lambda_j)^2},  \text{ for } i\neq j. 
%\end{align}
\end{lemma}

\begin{proof}
%By \eqref{eq-1.1} and the orthogonality of $U$, we have
%\begin{align*}
%	\sum_{k \le h} \dfrac{\partial \Phi_i}{\partial b_{kh}} \dfrac{\partial \Phi_j}{\partial b_{kh}}
%	&= \sum_{k<h} \dfrac{ \partial \tilde \Phi_i}{\partial X_{kh}} \dfrac{\partial \tilde\Phi_j}{\partial X_{kh}} 
%	+ 2 \sum_{k=1}^d \dfrac{\partial \tilde \Phi_i}{\partial X_{kk}} \dfrac{\partial \tilde\Phi_j}{\partial X_{kk}} \\
%	&= 4 \sum_{k<h} U_{ki} U_{hi} U_{kj} U_{hj} + 2 \sum_{k=1}^d U_{ki}^2 U_{kj}^2 \\
%	&= 2 \bigg( \sum_{k=1}^d U_{ki} U_{kj} \bigg)^2 = 2 \, \1_{[i=j]}.
%\end{align*}
%This proves \eqref{eq-2.2'}.

By \eqref{eq-1.2} and the orthogonality of $U$, we have
\begin{align*}
	\sum_{k \le h} \dfrac{\partial^2 \Phi_i}{\partial b_{kh}^2}
	&= \sum_{k < h} \dfrac{\partial^2 \tilde \Phi_i}{\partial X_{kh}^2} + 2\sum_{k=1}^d \dfrac{\partial^2 \tilde \Phi_i}{\partial X_{kk}^2} \nonumber \\
	&= 2\sum_{k < h}  \sum_{j:j \neq i} \dfrac{\left| U_{ki} U_{hj} + U_{hi} U_{kj} \right|^2}{\lambda_i - \lambda_j} + 4\sum_{k=1}^d  \sum_{j:j \neq i} \dfrac{\left| U_{ki} U_{kj} \right|^2}{\lambda_i - \lambda_j} \nonumber \\
	&= \sum_{j:j \neq i} \dfrac{\sum_{k, h} \left| U_{ki} U_{hj} + U_{hi} U_{kj} \right|^2}{\lambda_i - \lambda_j} = 2 \sum_{j:j \neq i} \dfrac{1}{\lambda_i - \lambda_j},
\end{align*}
where the last equality follows from the orthogonality of $U$ and Lemma \ref{lem1}. This proves \eqref{eq-2.2}.

Next, we show \eqref{eq-2.4'}. By the chain rule, we can write
\begin{align} \label{eq-4.3'}
	\sum_{k' \le h'} \dfrac{\partial^3 \Phi_i}{\partial b_{kh} \partial b_{k'h'}^2}
	= (\1_{[k<h]} + \sqrt{2} \, \1_{[k=h]}) \bigg( \sum_{k'<h'} \dfrac{\partial^3 \tilde \Phi_i}{\partial X_{kh} \partial X_{k'h'}^2} 
	+ 2\sum_{k'=1}^d \dfrac{\partial^3 \tilde \Phi_i}{\partial X_{kh} \partial X_{k'k'}^2} \bigg).
\end{align}
We choose the parameter $\theta = X_{kh}$ and $\xi = X_{h'k'}$ in \eqref{eq-lemma 1-3}. The terms with second order or third order derivative 
vanish and we only need to consider the terms with only the first order derivative. Note that for indices $1 \le p_1, p_2, q_1, q_2 \le d$ 
\begin{align} \label{eq-2.6''}
	& \sum_{k'<h'} \left( U^{\T} \dfrac{\partial X}{\partial X_{k'h'}} U \right)_{p_1p_2} \left( U^{\T} \dfrac{\partial X}{\partial X_{k'h'}} U \right)_{q_1q_2}
	+ 2 \sum_{k'=1}^d \left( U^{\T} \dfrac{\partial X}{\partial X_{k'k'}} U \right)_{p_1p_2} \left( U^{\T} \dfrac{\partial X}{\partial X_{k'k'}} U \right)_{q_1q_2} \nonumber \\
	&= \sum_{k'<h'} (U_{k'p_1}U_{h'p_2} + U_{h'p_1}U_{k'p_2}) (U_{k'q_1}U_{h'q_2} + U_{h'q_1}U_{k'q_2})
	+ 2\sum_{k'=1}^d U_{k'p_1}U_{k'p_2} U_{k'q_1}U_{k'q_2} \nonumber \\
	&= \bigg( \sum_{k'=1}^d U_{k'p_1}U_{k'q_1} \bigg) \bigg( \sum_{h'=1}^d U_{h'p_2}U_{h'q_2} \bigg) + \bigg( \sum_{k'=1}^d U_{k'p_1}U_{k'q_2} \bigg) 
	\bigg( \sum_{h'=1}^d U_{h'p_2}U_{h'q_1} \bigg) \nonumber \\
	&= \1_{[p_1=q_1]} \1_{[p_2=q_2]} + \1_{[p_1=q_2]} \1_{[p_2=q_1]}.
\end{align}
Now taking sum over $(k', h')$ for \eqref{eq-lemma 1-3} (i.e. taking sum over the non-zero terms including $U^\T\partial_\xi XU$), and 
applying \eqref{eq-2.6''}, we have
\begin{align*}
\sum_{k' \le h'} \dfrac{\partial^3 \Phi_i}{\partial b_{kh} \partial b_{k'h'}^2}
	=2\left(\1_{[k<h]} + \sqrt{2} \, \1_{[k=h]}\right)\sum_{j:j\neq i}\bigg(\frac{(U^\T \partial_\theta X)_{jj}}{(\lambda_i-\lambda_j)^2}- \frac{(U^\T \partial_\theta X)_{ii}}{(\lambda_i-\lambda_j)^2} \bigg). 
\end{align*}
This together with \eqref{e:A18} yields  the first equality of \eqref{eq-2.4'}.  The second equality of \eqref{eq-2.4'} now follows \eqref{eq-1.1}:
\begin{align*} 
	&2\left( 2 \, \1_{[k<h]} + \sqrt{2} \,\1_{[k=h]} \right) \sum_{j:j\neq i} \dfrac{U_{kj} U_{hj} - U_{ki} U_{hi}}{(\lambda_i - \lambda_j)^2} \nonumber \\
	&=2 \left( \1_{[k<h]} + \sqrt{2} \, \1_{[k=h]} \right)  \sum_{j:j\neq i} \dfrac{1}{(\lambda_i - \lambda_j)^2} \bigg( \dfrac{\partial \tilde \Phi_j}
	{\partial X_{kh}} - \dfrac{\partial\tilde\Phi_i}{\partial X_{kh}} \bigg) \nonumber \\
	&= 2 \sum_{j:j\neq i} \dfrac{1}{(\lambda_i - \lambda_j)^2} \left( \dfrac{\partial \Phi_j}{\partial b_{kh}} - \dfrac{\partial \Phi_i}{\partial b_{kh}} \right). 
\end{align*}
This proves \eqref{eq-2.4'}.

Now we show \eqref{eq-2.11}. Note that for $i \neq j$,
\begin{align} \label{eq-2.5}
	\sum_{k \le h} \dfrac{\partial^2\Psi_{ij}}{\partial b_{kh}^2}
	&= \sum_{k \le h} \dfrac{\partial}{\partial b_{kh}} \left( - \Psi_{ij}^{2} \dfrac{\partial (\Phi_i - \Phi_j)}{\partial b_{kh}} \right) \nonumber \\
	&= \sum_{k \le h} 2\Psi_{ij}^{3} \left( \dfrac{\partial (\Phi_i - \Phi_j)}{\partial b_{kh}} \right)^2
	- \sum_{k \le h}  \Psi_{ij}^{2} \dfrac{\partial^2 (\Phi_i - \Phi_j)}{\partial b_{kh}^2} .
\end{align}
For the first term of \eqref{eq-2.5}, by \eqref{eq-1.1} and the orthogonality of the columns of $U$, for $i \neq j$, we have
\begin{align} \label{eq-2.6}
	\sum_{k \le h} 2\Psi_{ij}^{3} \bigg( \dfrac{\partial (\Phi_i - \Phi_j)}{\partial b_{kh}} \bigg)^2
	&= \dfrac{2}{(\Phi_i - \Phi_j)^3} \left( \sum_{k<h} \left( \dfrac{\partial (\Phi_i - \Phi_j)}{\partial b_{kh}} \right)^2 
	+ \sum_{k=1}^d \left( \dfrac{\partial (\Phi_i - \Phi_j)}{\partial b_{kk}} \right)^2 \right) \nonumber \\
	&= \dfrac{2}{(\Phi_i - \Phi_j)^3} \left( \sum_{k<h} \bigg( \dfrac{\partial (\tilde \Phi_i - \tilde\Phi_j)}{\partial X_{kh}} \bigg)^2
	+ 2\sum_{k=1}^d \bigg( \dfrac{\partial (\tilde\Phi_i - \tilde\Phi_j)}{\partial X_{kk}} \bigg)^2 \right) \nonumber \\
	&= \dfrac{2}{(\Phi_i - \Phi_j)^3} \left( 4 \sum_{k<h} \left( U_{ki}U_{hi} - U_{kj}U_{hj} \right)^2 
	+ 2 \sum_{k=1}^d \left( U_{ki}^2 - U_{kj}^2 \right)^2 \right) \nonumber \\
	&= \dfrac{4}{(\Phi_i - \Phi_j)^3 } \sum_{k,h=1}^d \left( U_{ki}U_{hi} - U_{kj}U_{hj} \right)^2 \nonumber \\
	&= \dfrac{8}{(\Phi_i - \Phi_j)^3}
	= \dfrac{8}{(\lambda_i - \lambda_j)^3},
\end{align}
where the last step follows from Lemma \ref{lem1}.

For the second term of \eqref{eq-2.5}, we have
\begin{align} \label{eq-2.7}
	\sum_{k \le h}  \Psi_{ij}^{2} \dfrac{\partial^2 (\Phi_i - \Phi_j)}{\partial b_{kh}^2} 
	&= \dfrac{1}{(\Phi_i - \Phi_j)^2} \left( \sum_{k<h} \dfrac{\partial^2 \Phi_i}{\partial b_{kh}^2} 
	+ \sum_{k=1}^d \dfrac{\partial^2 \Phi_i}{\partial b_{kk}^2} - \sum_{k<h} \dfrac{\partial^2 \Phi_j}{\partial b_{kh}^2} 
	- \sum_{k=1}^d \dfrac{\partial^2 \Phi_j}{\partial b_{kk}^2} \right) \nonumber \\
	&= \dfrac{1}{(\Phi_i - \Phi_j)^2} \left( \sum_{k<h} \dfrac{\partial^2 \tilde\Phi_i}{\partial X_{kh}^2} 
	+ 2 \sum_{k=1}^d \dfrac{\partial^2 \tilde\Phi_i}{\partial X_{kk}^2} - \sum_{k<h} \dfrac{\partial^2 \tilde\Phi_j}{\partial X_{kh}^2} 
	- 2\sum_{k=1}^d \dfrac{\partial^2 \tilde\Phi_j}{\partial X_{kk}^2} \right).
\end{align}
By \eqref{eq-1.2}, the orthogonality of the columns of $U$, and Lemma \ref{lem1}, for $i \neq j$, we have
\begin{align} \label{eq-2.8}
	\sum_{k<h} \dfrac{\partial^2\tilde \Phi_i}{\partial X_{kh}^2}
	+ 2 \sum_{k=1}^d \dfrac{\partial^2 \tilde\Phi_i}{\partial X_{kk}^2}
	&= \sum_{k<h} 2 \sum_{l:l \neq i} \dfrac{\left| U_{ki}U_{hl}
		+ U_{hi}U_{kl}\right|^2}{\lambda_i - \lambda_l}
	+ 2 \sum_{k=1}^d 2 \sum_{l:l \neq i} \dfrac{\left| U_{ki}U_{kl} \right|^2}{\lambda_i - \lambda_l} \nonumber \\
	&= \sum_{l:l \neq i} \dfrac{2\sum_{k<h} \left| U_{ki}U_{hl} + U_{hi} U_{kl} \right|^2 
		+ 4 \sum_k \left| U_{ki}U_{kl} \right|^2}{\lambda_i - \lambda_l} \nonumber \\
	&= \sum_{l:l \neq i} \dfrac{\sum_{k,h}  \left| U_{ki} U_{hl} + U_{hi} U_{kl} \right|^2}{\lambda_i - \lambda_l} 
	= \sum_{l:l \neq i} \dfrac{2}{\lambda_i - \lambda_l}.
\end{align}
Similarly, we have
\begin{align} \label{eq-2.9}
	\sum_{k<h} \dfrac{\partial^2 \tilde\Phi_j}{\partial X_{kh}^2} + 2 \sum_{k=1}^d \dfrac{\partial^2 \tilde\Phi_j}{\partial X_{kk}^2}
	= \sum_{l:l \neq j} \dfrac{2}{\lambda_j - \lambda_l}.
\end{align}
Putting \eqref{eq-2.8} and \eqref{eq-2.9} to \eqref{eq-2.7} yields that the second term of \eqref{eq-2.5} now is
\begin{align} \label{eq-2.10}
	\sum_{k \le h}  \Psi_{ij}^{2} \dfrac{\partial^2 (\Phi_i - \Phi_j)}{\partial b_{kh}^2} 
	&= \dfrac{1}{(\lambda_i - \lambda_j)^2} \Bigg( \sum_{l:l \neq i} \dfrac{2}{\lambda_i - \lambda_l}
	- \sum_{l:l \neq j} \dfrac{2}{\lambda_j- \lambda_l} \Bigg).
\end{align}
By substituting \eqref{eq-2.6} and \eqref{eq-2.10} into \eqref{eq-2.5}, we obtain
\begin{align*}
	\sum_{k \le h} \dfrac{\partial^2 \Psi_{ij}}{\partial b_{kh}^2}
	&= \dfrac{8}{(\lambda_i - \lambda_j)^3}
	- \dfrac{1}{(\lambda_i - \lambda_j)^2} \Bigg( \sum_{l:l \neq i} \dfrac{2}{\lambda_i - \lambda_l} 
	- \sum_{l:l \neq j} \dfrac{2}{\lambda_j - \lambda_l} \Bigg) \nonumber \\
	&= \dfrac{4}{(\lambda_i - \lambda_j)^3}
	- \dfrac{1}{(\lambda_i - \lambda_j)^2} \Bigg( \sum_{l:l \neq i,j} \dfrac{2}{\lambda_i - \lambda_l} 
	- \sum_{l:l \neq i,j} \dfrac{2}{\lambda_j - \lambda_l} \Bigg) \nonumber \\
	&= \dfrac{4}{(\lambda_i - \lambda_j)^3}
	- \dfrac{1}{(\lambda_i - \lambda_j)^2} \sum_{l:l \neq i,j} \dfrac{2\left( \lambda_j - \lambda_i \right)}
	{\left( \lambda_i- \lambda_l \right) \left( \lambda_j - \lambda_l \right)} \nonumber \\
	&= \dfrac{4}{(\lambda_i - \lambda_j)^3}
	+ \dfrac{1}{(\lambda_i - \lambda_j)} \sum_{l:l \neq i,j} \dfrac{2}{\left( \lambda_i - \lambda_l \right) \left( \lambda_j - \lambda_l \right)}.
\end{align*}
This proves \eqref{eq-2.11}. 
%
%Finally,  noting that for $i\neq j$:
%\begin{align*}
%	\sum_{k \le h} \dfrac{\partial \Phi_i}{\partial b_{kh}} \dfrac{\partial \Psi_{ij}}{\partial b_{kh}}
%	&= - \dfrac{1}{(\lambda_i-\lambda_j)^2} \sum_{k \le h} \dfrac{\partial \Phi_i}{\partial b_{kh}} \dfrac{\partial (\Phi_i - \Phi_j)}{\partial b_{kh}} \nonumber \\
%	&= - \dfrac{1}{(\lambda_i-\lambda_j)^2} \Bigg( \sum_{k \le h} \left( \dfrac{\partial \Phi_i}{\partial b_{kh}} \right)^2 - \sum_{k \le h} \dfrac{\partial \Phi_i}{\partial b_{kh}} \dfrac{\partial \Phi_j}{\partial b_{kh}} \Bigg) \nonumber \\
%	&= - \dfrac{2}{(\lambda_i-\lambda_j)^2}.
%\end{align*}
%Hence, identity \eqref{eq-A.25}  follows from \eqref{eq-2.2'}.
\end{proof}

\bibliographystyle{plain}
\bibliography{Reference}

\end{document}